\documentclass[reqno]{amsart}

\usepackage{amssymb}

\usepackage{hyperref}

\usepackage{stackengine}

\allowdisplaybreaks

\newtheorem{theorem}{Theorem}[section]
\newtheorem{lemma}[theorem]{Lemma}
\newtheorem{proposition}[theorem]{Proposition}

\theoremstyle{definition}
\newtheorem{definition}[theorem]{Definition}

\DeclareMathOperator*{\supp}{supp}

\DeclareMathOperator{\sgn}{sgn}

\newcommand{\noi}{\noindent}
\newcommand{\Z}{\mathbb{Z}}
\newcommand{\R}{\mathbb{R}}
\newcommand{\C}{\mathbb{C}}
\newcommand{\Q}{\mathbb{Q}}

\newcommand{\NB}{\mathbb{N}}

\newcommand{\weakstar}{\mathrel{\ensurestackMath{\stackon[1pt]{\rightharpoonup}{\scriptstyle\ast}}}}

\newcommand{\F}{\mathcal{F}}

\newcommand{\al}{\alpha}

\newcommand{\dl}{\delta}

\newcommand{\eps}{\varepsilon}

\newcommand{\s}{\sigma}
\newcommand{\ft}{\widehat}

\newcommand{\dx}{\partial_x}

\newcommand{\qtq}[1]{\quad\text{#1}\quad}

\newcommand{\jb}[1]
{\langle #1 \rangle}

\renewcommand{\l}{\ell}

\newcommand{\les}{\lesssim}
\newcommand{\ges}{\gtrsim}

\numberwithin{equation}{section}
\numberwithin{theorem}{section}

\begin{document}

\title[Bounded solutions of KdV]{Bounded solutions of KdV: uniqueness\\and the loss of almost periodicity}

\author{Andreia Chapouto}
\address{
Department of Mathematics\\
University of California\\Los Angeles\\CA 90095\\USA}
\email{chapouto@math.ucla.edu}

\author{Rowan Killip}
\address{
Department of Mathematics\\
University of California\\Los Angeles\\CA 90095\\USA}
\email{killip@math.ucla.edu}

\author{Monica Vi\c{s}an}
\address{
Department of Mathematics\\
University of California\\Los Angeles\\CA 90095\\USA}
\email{visan@math.ucla.edu}

\begin{abstract}
We address two pressing questions in the theory of the Korteweg--de Vries (KdV) equation. First, we show the uniqueness of solutions to KdV that are merely bounded, without any further decay, regularity, periodicity, or almost periodicity assumptions. The second question, emphasized by Deift \cite{Deift08,Deift17}, regards whether almost periodic initial data leads to almost periodic solutions to KdV. Building on the new observation that this is false for the Airy equation, we construct an example of almost periodic initial data whose KdV evolution remains bounded, but fails to be almost periodic at a later time.  Our uniqueness result ensures that the solution constructed is the unique development of this initial data.
\end{abstract}


\maketitle

\section{Introduction}

We study solutions to the Korteweg--de Vries equation (KdV) 
\begin{equation} \label{kdv}\tag{KdV}
	\tfrac{d}{dt} q = - q''' + 3 (q^2)'  ,
\end{equation} 
which describes the evolution of a real-valued function defined on the line $\R$.  Here primes denote spatial derivatives.

We are interested in studying solutions to \eqref{kdv} that are merely bounded, without further regularity or decay assumptions.  This class includes a multitude of different spatial profiles of enduring interest, including step-like solutions \cite{MR3548255,MR0139371,MR0748367,MR1263128,
MR3071444, GP74,MR4356987,MR2439485,MR3229496}, quasi- and almost periodic solutions \cite{BDGL18,DG16,Deift08,Deift17,Egorova,EVY,Tsugawa12}, as well as soliton gases \cite{MR3261206,DZZ16,MR4410140,MR4259375,MR1098341,ZDZ16}.

Our first objective is to show that such bounded solutions are uniquely determined by their initial data, without auxiliary conditions.  This is known as \emph{unconditional uniqueness}.  This term was coined by Kato in the paper \cite{Kato95}, which studied nonlinear Schr\"odinger equations in $H^s(\R)$ spaces.  The unconditionality of Kato's results meant precisely that he could prove that the solutions were unique in the space $C_t H^s_x$.  By comparison, Kato explains, solutions that are constructed via contraction mapping in Strichartz spaces, are only guaranteed to be unique amongst competitors that also have finite Strichartz norm.  

For initial data with more interesting spatial asymptotics, solutions are constructed, perforce, under the presumption that they will maintain the same spatial asymptotics.  Indeed, even for periodic initial data, solutions are constructed within the class of functions with the same spatial period and unconditional uniqueness has traditionally been interpreted in this sense too.  Is it really necessary to enforce this restriction? Or does it follow directly from \eqref{kdv}?  Similarly, on the basis of physical intuition, we expect initial data decaying at infinity  (say, in the sense of belonging to an $H^s(\R)$ space) to yield only solutions that likewise decay at infinity.  But can we prove this?  What assumptions are necessary?

Without examples of wild behaviour, the questions of the previous paragraph may seem foolish.  However, such examples do exist!  Both \cite{arXiv:0503366} and \cite{MR0988885} construct nonzero solutions to \eqref{kdv} with zero initial data.  In the case of \cite{MR0988885}, the solutions are infinitely smooth and defined on a narrow spacetime region around the set $t\equiv 0$.  Their jumping-off point for this construction is the existence of a globally defined smooth nonzero solution to the Airy equation, $\partial_t q = - q'''$, that vanishes for all $t\leq0$.

By contrast, the solutions constructed in \cite{arXiv:0503366} are very irregular, namely, $C_tH^s$ with $s<0$; correspondingly, the definition of solution employed in that paper is rather subtle.  These solutions are also periodic in space.  As the zero function is periodic with any period one chooses, Christ's solutions provide an example where the period of the solution does not coincide with that of the initial data.  They may also be regarded as solutions whose initial data are rapidly decreasing, but then suddenly, are not.

Before turning to our principal uniqueness result, namely, Theorem~\ref{th:unique}, we must first pause to make the notion of a bounded solution precise. 

Given an open interval $I\subseteq \R$, a bounded measurable function $q:I\times\R\to\R$ is said to be a \emph{distributional solution} to \eqref{kdv} if
\begin{equation}\label{dsol}
\iint \bigl[ \partial_t \phi(t,x) + \phi'''(t,x) \bigr] q(t,x)\,dx\,dt = 
		3 \iint \phi'(t,x) q(t,x)^2 \, dt \,dx, 
\end{equation}
for every $\phi \in C^\infty_c(I\times\R)$.  Evidently, such a solution can be modified on any spacetime null set without affecting its status as a distributional solution; this includes any fixed-time slice!

To remove this ambiguity, we may demand that $q(t,x)$ agrees with its spacetime Lebesgue values (where they exist); this only affects the values on a spacetime null set.  With this change, \eqref{dsol} guarantees that
\begin{equation*}
\int_{s_1}^{s_2}\!\!\int q(t,x) \psi'''(x) - 3q(t,x)^2\psi'(x) \,dx\,dt = \int [q(s_2,x)-q(s_1,x)]\psi(x)\,dx
\end{equation*} 
for every choice of $\psi\in C^\infty_c(\R)$ and all $s_1<s_2$ belonging to the time interval $I$.  This in turn demonstrates that
\begin{equation}\label{w*cts}
t\mapsto \int \psi(x) q(t,x) \,dx \quad\text{is continuous for all $\psi\in C^\infty_c(\R)$.}
\end{equation}
As $q$ is already assumed to be bounded, \eqref{w*cts} is equivalent to the statement that $t\mapsto q(t,x)$ is continuous into $L^\infty$ endowed with the weak-$*$ topology.  This line of reasoning justifies our preferred notion of solution:

\begin{definition}\label{def:solution}
Given an open interval $I\subseteq \R$, a bounded measurable function $q:I\times\R\to\R$ is a \emph{bounded solution to \eqref{kdv}} if it is a distributional solution and weak-$*$ continuous, which is to say \eqref{dsol} and \eqref{w*cts} hold.
\end{definition}

One should resist the temptation to adopt the norm topology on $L^\infty$ here.  First, it would lead to a more restrictive notion of solution and so weaken our uniqueness claim.  Secondly, it is also rather unnatural for PDE problems.  For example, the time-dependent characteristic function $\chi_{[t,\infty)}$ (which solves the simple transport equation $\partial_t q + q' =0$) is weak-$*$ continuous, but not norm continuous.

For the consideration of bounded solutions, we need only consider the weak\nobreakdash-$*$ topology on closed balls in $L^\infty$.  These topological spaces are completely metrizable and compact, which is a very comfortable setting in which to do analysis.

Continuity ensures a meaningful connection between the solution and its initial data.  The question of unconditional uniqueness is whether each initial data admits at most one continuous development.  This we answer in the affirmative for bounded solutions to \eqref{kdv}:

\begin{theorem}\label{th:unique}
Let $q_1$ and $q_2$ be bounded solutions to \eqref{kdv}, both defined on some open interval $I\subseteq \R$.  If $q_1(t_0) = q_2(t_0)$ as elements of $L^\infty(\R)$ for a single $t_0\in I$, then $q_1(t) = q_2(t)$ for all $t\in I$.
\end{theorem}

We know of no prior unconditional well-posedness results for \eqref{kdv} under mere boundedness constraints, no matter how many space or time derivatives are assumed bounded.

The uniqueness questions that have received the most attention are those related to $C_t^{ }H^s_x$ solutions on the line $\R$ and on the circle $\R/\Z$.  (Working on the circle is equivalent to studying periodic solutions with an enforced period.)   After reviewing this, we will discuss recent results of \cite{BDGL18,MR4122653}, which provide the only other unconditional uniqueness results that we know of.  These consider certain classes of almost periodic initial data, motivated by questions posed by Deift \cite{Deift08,Deift17} on the spacetime almost periodicity of solutions to \eqref{kdv} with almost periodic initial data. This discussion will lead naturally to the second main contribution of this paper, namely, Theorem~\ref{th:deift}, which demonstrates the existence of solutions whose initial data is almost periodic but whose later evolution is not.

Early results on the well-posedness problem for \eqref{kdv} focused on initial data in $H^s$ spaces.  All employ the same uniqueness argument, which we will now explain.  For any pair of classical solutions $q_1$ and $q_2$ to \eqref{kdv}, we have
\begin{align}\label{micro ag}
\partial_t (q_1-q_2)^2 =& - \partial_x^3 (q_1-q_2)^2 + 3 \partial_x \bigl\{ (q_1'-q_2')^2 \bigr\} + 3 (q_1'+q_2')(q_1-q_2)^2 \notag\\
	& + 3 \partial_x \bigl\{  (q_1+q_2)(q_1-q_2)^2 \bigr\}  .
\end{align}
By integrating over the whole space, we find that
\begin{align}\label{macro ag}
\partial_t \int (q_1-q_2)^2\,dx \leq 3 \Bigl[\| q_1' \|_{L^\infty} + \| q_2' \|_{L^\infty} \Bigr] \int (q_1-q_2)^2\,dx .
\end{align}
This type of inequality allows one to deduce uniqueness via Gronwall's inequality.  In this case, it would require that both solutions belong to $C_t^{ }L^2_x$ and that $q_1'$ and $q_2'$ belong to $L^1_tL^\infty_x$.  Consequently, this argument shows unconditional uniqueness in $C_t H^s(\R)$ and $C_t H^s(\R/\Z)$ for any $s>3/2$.

The proof of Theorem~\ref{th:unique} will also, ultimately, employ a Gronwall-type argument.  The fundamental difficulty in such an argument is finding an effective notion of the `distance' between two solutions.  It is crucial that one can control the increment of the distance in terms of itself, as in \eqref{macro ag}.  This is a daunting task even for Schwartz-class solutions; the elegant simplicity of \eqref{micro ag} belies essential algebraic miracles.  Nevertheless, we have found another notion of distance with this miraculous property; this can be seen by setting $\psi\equiv 1$ and $\tilde F_1 \equiv \tilde F_2 \equiv 0$ in \eqref{gronwall}.  Moreover, the rate of exponential growth of this distance is controlled for solutions that are merely bounded. 

There is a second obstacle that we must also overcome: As we wish to treat solutions without spatial decay, we must adopt some localized notion of distance.  This is in direct conflict with the fact that we are considering a dispersive equation: high-frequency waves travel fast and so may lead to rapid inflation of the difference in localized norms.  The concomitant loss of derivatives is evident already from the second term in RHS\eqref{micro ag}.  This phenomenon likewise manifests in the second spacetime integral in RHS\eqref{gronwall}.  In Section~\ref{sec:3} we present a means of overcoming this loss by exploiting the fact that one derivative falls on the localizing weight.
 
It is now known that solutions of \eqref{kdv} are unconditionally unique in $C_t^{ } L^2_x$ both on the line and on the circle.  In the line case, this was shown by Zhou in \cite{MR1440304}.  The first observation is that $C_t^{ } L^2_x$ solutions automatically belong to certain $X^{s,b}$ spaces.  The Duhamel formula is then used to show that the $X^{s,b}$ norm of a difference of solutions does not exceed a small multiple of itself; this guarantees uniqueness.

Unconditional uniqueness in $C_t^{ }L^2_x(\R/\Z)$ was proved in \cite{MR2789490}.  By making a bijective change of unknown in the spirit of Birkhoff normal form, the authors reduce \eqref{kdv} to an integral equation that can be solved by contraction mapping in $C_t^{ }L^2_x(\R/\Z)$ without any auxiliary norms. Naturally, this yields $C_t^{ }L^2_x(\R/\Z)$ uniqueness.

Any attempt to prove unconditional uniqueness in $C_t^{ }H^s_x$ for $s<0$ must address a very real question: What does it mean for such a distribution to be a solution of \eqref{kdv}?  One cannot simply square such distributions!   Christ's work \cite{arXiv:0503366} gives one answer to this question and shows that it leads to nonuniqueness.  A competing notion, named green solutions, was introduced in \cite{KMV} specifically to give meaning to the unconditional uniqueness question for the white-noise solutions constructed therein, as well as the $C_t^{ }H^{-1}_x$ solutions constructed in \cite{MR2267286,KV19}.  These questions remain open.

Let us turn now to the case of quasi- and almost periodic initial data.  Recall that a function $f:\R\to\R$ is called \emph{quasiperiodic} if there is a finite dimensional torus $\R^n/\Z^n$, a vector $\omega\in\R^n$, and a continuous function $F:\R^n/\Z^n\to\R$ so that $f$ may be represented as
\begin{align}\label{quasi-p}
f(x) = F( x\omega + \Z^n) \qtq{for all} x\in\R.
\end{align}
Conventionally, one chooses $\omega$ so that its entries are linearly independent over $\Q$, for otherwise, one may just use a lower-dimensional torus.

The notion of an almost periodic function may be regarded as the $n=\infty$ case of the above.  We prefer to present the original definition (cf. \cite[\S44]{Bohr:book}).  A number $\ell\in\R$ is called an $\eps$ almost period of the function $f:\R\to\R$ if
\begin{align}\label{almost period}
\bigl\| f(x+\ell) - f(x) \bigr\|_{L^\infty(\R)} < \eps.
\end{align}
A function $f:\R\to\R$ is said to be \emph{almost periodic} if it is continuous and for every $\eps>0$ there is an $L_\eps>0$ so that every interval of length $L_\eps$ in $\R$ contains at least one $\eps$ almost period.  An equivalent characterization is given by Bohr's Fundamental Theorem of Almost Periodic Functions (cf. \cite[\S44]{Bohr:book}): a function $f:\R\to\R$ is almost periodic if and only if it can written as the uniform limit of finite trigonometric sums (with unrestricted frequencies).

A great deal of work has been devoted to the study of \eqref{kdv} with quasi- and almost periodic initial data, both as an end unto itself and through its connection, via the Lax-pair formulation, to the quantum mechanics of one-dimensional quasicrystals.  Naturally, the uniqueness of such solutions was investigated as an integral part of this program.  Nevertheless, we know of only one uniqueness result that may reasonably be categorized as unconditional, namely, that of \cite{BDGL18}.  This paper constructs solutions for almost periodic initial data under certain restrictions on the Schr\"odinger operator with this potential: the spectrum must be absolutely continuous, reflectionless, and satisfy Craig-type conditions.  The authors prove that there is only one solution with this initial data for which $q$, $q'$, and $\partial_t q$ remain bounded.  (By virtue of the equation, $q'''$ and $q''$ also remain bounded.)  Central to this achievement is the proof (building on ideas from \cite{Rybkin08}) that under these assumptions, all such solutions must retain the spectral properties imposed on the initial data.

The subsequent paper \cite{MR4122653} extends \cite{BDGL18} in two ways: it reduces the regularity requirements to boundedness of two spatial derivatives and it extends the uniqueness result to higher order flows in the \eqref{kdv} hierarchy (under stronger Craig-type conditions).  The authors of \cite{MR4122653} also observe that these techniques yield a new result for the case of periodic initial data, namely, classical solutions with periodic initial data for which $q$ and $q''$ remain bounded must remain periodic.

While we contend that Theorem~\ref{th:unique} provides a definitive resolution of the uniqueness question for almost periodic initial data, a great deal remains to be done regarding the existence question.  It is indicative of the difficulty of this problem that there are no known robust methods for obtaining a priori bounds on the solution.  The well-known conservation laws associated to KdV, including momentum and energy, are simply useless because they are all infinite.  The fundamental enemy is that the infinite momentum, for example, may all pile up in one place!  In Section~\ref{sec:stepanov} we give an example of almost periodic initial data for which precisely this happens, even under the simpler Airy dynamics:
\begin{equation}\label{Airy}
	\tfrac{d}{dt} q = - q''' .
\end{equation}

Going beyond well-posedness, it is natural to ask about the long time behaviour of solutions.  The numerical investigations \cite{KruskalZabusky} of the periodic case by Kruskal and Zabusky, which first thrust \eqref{kdv} into the limelight, already showed almost recurrence of the initial state after a short time.  Subsequent numerics and investigation of finite-gap solutions (cf. \cite{MR0382877,MR0404889}) further solidified the prediction that periodic solutions evolve almost periodically in time.  For smooth solutions, this was proved by McKean and Trubowitz in \cite{MT76}, with subsequent extension to $L^2$ in \cite{Bo93-2} and then to $H^{-1}$ in \cite{MR2267286}.

In \cite{Deift08}, and again in \cite{Deift17}, Deift made the conjecture that initial data that is almost periodic in space leads to solutions that are almost periodic in time.   We know of several rigorous results that support this thesis:  In \cite{Egorova}, Egorova proves that this is true for certain limit periodic initial data; it is required that the initial data may be extremely well approximated by periodic functions (faster than any exponential of the period).

The more recent work \cite{BDGL18} demonstrates almost periodicity of the solution for a disjoint class of initial conditions.  Theorem~1 of that paper resolves the conjecture for small real analytic quasiperiodic initial data with Diophantine frequency vector.  Their second theorem covers a broader class of initial data determined by the spectral theory of the Schr\"odinger operator associated to the initial data.  At least currently, these spectral assumptions can only be verified for similarly narrow classes of initial data.  These spectral conditions were further relaxed in \cite{EVY}.  The analogue of the Deift conjecture for higher flows in the integrable hierarchy is treated in \cite{EVY,MR4122653}.

Setting aside the formidable analytical difficulties that must be overcome, the approach pursued in these papers gives rise to an intuitive appreciation of why the Deift conjecture ought to be true:   The magic of complete integrability suggests that for almost periodic initial data, \eqref{kdv} together will all its commuting flows can be conjugated to commuting (Dubrovin) translations on a compact Abelian group (a  product of tori indexed by the spectral gaps).   The conjugation mapping provides a function that maps each point on the group to the value of the corresponding $q$ at the spatial origin.  By exploiting the fact that translation is one of the commuting flows, this sampling function allows one to reconstruct $q$ at all spatial points.

Taking this line of reasoning to its natural conclusion leads one to predict that the solution to \eqref{kdv} with almost periodic initial data is in fact an almost periodic function of \emph{spacetime}.  In particular, this promises that $t\mapsto q(t,x)$ is almost periodic for every $x$ and likewise, $x\mapsto q(t,x)$ is almost periodic for every $t$.  This strong formulation of the Deift conjecture is shown to hold in each of the papers \cite{BDGL18,Egorova,EVY,MR4122653}.  It is a great triumph and demonstrates the thorough understanding they have achieved for their classes of initial data.  However, as we will argue, it appears not to be the whole truth.

In doubting the unconditional veracity of the Deift conjecture, we are preceded by \cite{DLVY21}, which outlined an extensive program for potentially building a counterexample.   We will take a very different (and much simpler) approach; nevertheless, it was Damanik's lecture \cite{Damanik_ICERM} that stimulated our consideration of this problem.  There, it is also highlighted that continuity of the mapping conjugating \eqref{kdv} to group translations is pivotal to the reasoning laid out above; correspondingly, a keystone in their program is the construction of a regime in which there is such a conjugation but it is not continuous.

Our own investigations began with a seemingly innocuous question: Is the analogue of the Deift conjecture true for the Airy equation \eqref{Airy}?  This is something of an ill-posed problem: no definitive class of initial data has been specified nor has a metric been agreed upon.  What we will demonstrate is this: an originalist interpretation leads to the conclusion that it is false.  Building on this, we will demonstrate that the strong formulation of the Deift conjecture for \eqref{kdv} is likewise false:

\begin{theorem}\label{th:deift}
There is a bounded solution $q:[-T,T]\times\R\to\R$ of \eqref{kdv} with almost periodic initial data for which $x\mapsto q(t_0,x)$ is \emph{not} almost periodic at some time $t_0\in[-T,T]$. 
\end{theorem}

Here bounded solution means in the sense of Definition~\ref{def:solution}.  By Theorem~\ref{th:unique}, there can be no better-behaved solution with this same initial data.  

Let us first consider the Airy equation \eqref{Airy}.  When the initial data is a trigonometric sum with $\ell^1$ coefficients (and unrestricted frequencies), it is elementary to see that the corresponding solution is an almost periodic function of spacetime.  Evidently, we must go beyond this class.

The coefficients of our example will be weak-$\ell^1$.   This class includes discontinuous functions, for example, the $2\pi$-periodic square wave
\begin{align}\label{sq wave}
\operatorname{sq}(x)  = \sgn \big( \sin(x) \big) = \sum_{\xi \in \Z \text{ odd }} \tfrac{2}{\pi i \xi} e^{i \xi x }.
\end{align} 
Although this function is periodic, it is not almost periodic because it is not continuous.  While we would never dream of demanding continuity of a periodic function, Bohr insisted on this condition with good reason.  It is obvious that approximation by trigonometric polynomials would fail if one omitted continuity.  Bohr saw deeper: without continuity, the sum of two almost periodic functions need not be almost periodic!  Consider
\begin{align}\label{f is square}
f(x) = \operatorname{sq}(x) + \operatorname{sq}(\alpha x) \qtq{with} \alpha\in\R\setminus\Q .
\end{align}
Evidently, each summand is periodic. However, the sum is devoid of meaningful almost periods:  if $\eps\leq1$ then \eqref{almost period} holds if and only if $\ell=0$.

We may disprove the Deift conjecture (as formulated above) for the Airy equation by exhibiting almost periodic initial data whose later evolution coincides with \eqref{f is square}.  For expository reasons, it will be better to use \eqref{f is square} as initial data and then verify that at some other time this solution is almost periodic in space.  Note that both Airy and \eqref{kdv} are time translation invariant; they also admit the same time-reversal symmetry: $q(t,x)\mapsto q(-t,-x)$.

As a linear equation, the Airy evolution satisfies the principle of superposition. Correspondingly, to understand the solution $q(t,x)$ with initial data \eqref{f is square}, we need only study the evolution $w$ of a single periodic square wave: 
\begin{equation}\label{E:Airy sq}
\partial_t w = - w''' \qtq{with} w(0,x) = \operatorname{sq}(x)
\end{equation}
for then $q(t,x)=w(t,x) + w(\alpha^3t,\alpha x)$.  As it is periodic, the solution $w$ may be understood through the highly-developed theory of exponential sums over $\Z$ with polynomial phases.  The fine estimates that we will need are already known; they (and indeed much more) were proved by Oskolkov in \cite{oskolkov}; see Theorem~\ref{th:airy}.

Oskolkov proves that $x\mapsto w(t,x)$ is continuous whenever $t/2\pi$ is irrational.  Thus, if both $t/2\pi$ and $\alpha t/2\pi$ are irrational, then $q(t,x)$ is the sum of two continuous periodic functions and consequently almost periodic (it is even quasiperiodic!).  This proves the analogue of Theorem~\ref{th:deift} for the Airy equation.  

To prove Theorem~\ref{th:deift}, we wish to consider the solution of \eqref{kdv} with the initial data \eqref{f is square}.  But does such a solution even exist?  It is indicative of the subtlety of almost periodic initial data  that this is a nontrivial problem.  By developing a variant of the $X^{s,b}$ theory adapted to quasiperiodic (in the Stepanov, not Bohr, sense) initial data, Tsugawa constructs local-in-time solutions to \eqref{kdv} for certain types of initial data.  This result does apply to the initial data \eqref{f is square}; however, it does not automatically guarantee that the solution is bounded, neither in the naive sense, nor in the more precise sense of Definition~\ref{def:solution}.  This we will need to prove ourselves.

In order to prove Theorem~\ref{th:deift} by building on our observations for the Airy equation, the key step is to show that the solution with initial data \eqref{f is square} is continuous at some later (or earlier) time $t_1$.  To do this, we demonstrate a suitable nonlinear smoothing effect.  This is a broadly observed  phenomenon that the difference between a solution to a nonlinear dispersive equation and the linear evolution with the same initial data is smoother than either of the two solutions individually.  When solutions are constructed by the traditional combination of contraction mapping and the Duhamel formula, it is evident that this difference is smaller than either solution.  To exhibit nonlinear smoothing one must also exhibit a little extra smoothing from the spacetime integral.

This nonlinear smoothing effect is thoroughly discussed in the periodic setting in \cite{MR3559154}.  Evidently, we need to demonstrate nonlinear smoothing in the quasiperiodic setting, which we do not believe has been done previously.   This is the principal topic of Section~\ref{sec:quasi}, where we show that the difference between the \eqref{kdv} and Airy solutions is a continuous function in spacetime.  This is all that is needed to prove Theorem~\ref{th:deift}.   Given this specific goal, we strive for simplicity over generality and impose a Diophantine condition on the wavenumber $\alpha$ appearing in \eqref{f is square}.

Theorem~\ref{th:deift} focuses on just two times: one where the solution \emph{is} almost periodic and another when it is not.  It is natural to ask how it behaves for every time in $[-T,T]$.  This we can answer.  Our nonlinear smoothing result guarantees that $x\mapsto q(t,x)$ is almost periodic for a given time $t$ if and only if the corresponding linear solution is almost periodic.  This question in turn can be answered from the work of Oskolkov \cite{oskolkov}:  For the function $w$ defined in \eqref{E:Airy sq}, the mapping $x\mapsto w(t,x)$ is continuous if and only if $t/2\pi$ is irrational.  When $t/2\pi$ is rational, $x\mapsto w(t,x)$ has jump discontinuities --- indeed, it is piecewise constant!

This recurrence of discontinuities is evidently quite remarkable.  Indeed, a closer analysis shows that it is, in fact, the whole initial data that is being revisited!  This phenomenon was first observed by Talbot \cite{talbot} in optical experiments and so is known as the Talbot effect.  These experiments are better modeled by the linear Schr\"odinger equation, rather than the Airy equation; nevertheless, this does not significantly alter the underlying mathematics.  For a further, fuller discussion of the Talbot effect from both mathematical and physical points of view, see \cite{BerryKlein,BMS2022,MR3094557,MR3559154,OlverAMM}.

At this moment, we do not know whether the solution to \eqref{kdv} with initial data \eqref{f is square} exists globally in time.  However, this does not entirely preclude us from asking whether $t\mapsto q(t,0)$ is almost periodic.  For if the solution truly blows up, then this function is definitely not almost periodic.  On the other hand, our nonlinear smoothing estimate remains valid for as long as the solution persists in the function spaces of Theorem~\ref{th:lwp}.  In this way, we are lead to ask if $t\mapsto q(t,0)$ is almost periodic when $q$ is the solution to the Airy equation with initial data \eqref{f is square}.  It is not; see \cite[p. 390]{oskolkov}, where it is further explained that the answer would reverse if one considered the linear Schr\"odinger instead!

Given the mild nature of the loss of almost periodicity we exhibit to prove Theorem~\ref{th:deift}, it is natural to imagine that the veracity of the Deift conjecture might be restored if one simply relaxed the notion of almost periodicity so as to include the function \eqref{f is square}.  This relaxation comes at the cost: it enlarges the class of solutions that need to be understood. 

In Section~\ref{sec:stepanov}, we argue that relaxing the definition of almost periodicity is the wrong direction.  Concretely, we show that there is a Bohr almost periodic function whose Airy evolution undergoes an infinite concentration of $L^2$ norm in finite time.


\subsection*{Acknowledgements} R.K. was supported by NSF grant DMS--2154022 and M.V. by NSF grants DMS--1763074 and DMS--2054194.

\section{The Green's function}\label{sec:green}

This section is dedicated to the analysis of the Green's function associated to the Schr\"odinger operator 
\begin{align}\label{L op}
L:= - \dx^2 + q
\end{align} 
for potentials $q \in L^\infty(\R)$.

When $q\equiv 0$, the resolvent $R_0(\kappa) = (-\dx^2 + \kappa^2)^{-1}$ has integral kernel 
\begin{align*}
G_0(x,y;\kappa) = \tfrac{1}{2\kappa} e^{-\kappa|x-y|} \qtq{for all} \kappa>0.
\end{align*}

The Kato--Rellich Theorem guarantees that there exists a unique self-adjoint operator $L$ for $q\in L^\infty(\R)$. For $\kappa^2 \geq 4\|q\|_{L^\infty(\R)}$, the resolvent $R(\kappa) = (L + \kappa^2)^{-1}$ is given by the norm-convergent series expansion
\begin{equation*}
R(\kappa) = \sum_{\l=0}^\infty (-1)^\l \bigl(R_0(\kappa) q \bigr)^\l R_0(\kappa).
\end{equation*}

For our purposes, it is more convenient to eschew operator-theoretic considerations and work directly with the corresponding series expansion of the Green's function, whose terms take the form
\begin{align}\label{Neumann term}
\big< \dl_x , (R_0 q)^\l R_0\dl_y \big>
	= \int G_0(x,x_1) \bigg( \prod_{j=1}^{\l-1} q(x_j) G_0(x_j, x_{j+1}) \bigg) q (x_\l) G_0(x_\l, y) \, dx_1 \cdots dx_\l .
\end{align}

\begin{proposition}\label{prop:Green}
Let $q\in L^\infty(\R)$ and $\kappa^2\geq 4\|q\|_{L^\infty}$. The resolvent $R$ admits a continuous integral kernel defined by the absolutely convergent series
\begin{equation}\label{greens}
G(x,y;\kappa,q) = \sum_{\l=0}^\infty (-1)^\l \big< \dl_x , \bigl(R_0(\kappa) q \bigr)^\l R_0(\kappa)\dl_y \big>,
\end{equation}
which satisfies $G(x,y)=G(y,x)$ and 
\begin{equation}\label{G-upper}
|G(x,y) | \leq \tfrac{3}{4\kappa} e^{-\frac\kappa2|x-y|}. 
\end{equation}
Moreover, the diagonal Green's function $g(x;\kappa,q) := G(x,x; \kappa,q)$ satisfies
\begin{align}\label{g-bdd}
\tfrac{1}{4\kappa} \leq g(x) \leq \tfrac{3}{4\kappa} \quad\text{for all $x\in\R$}.
\end{align}
Finally, if $q_n\in L^\infty(\R)$ satisfy 
\begin{equation}\label{qn to q}
\| q_n\|_{L^\infty} \leq \| q \|_{L^\infty} \qtq{and} q_n \weakstar q \text{ in } L^\infty(\R) \text{ as } n\to\infty,
\end{equation}
then the corresponding Green's functions converge pointwise.
\end{proposition}

\begin{proof}
For each $\l\geq 0$, the kernel $\langle\dl_x , (R_0 q)^\l R_0 \dl_y\rangle$ of the operator $(R_0 q)^\l R_0$ is continuous in $(x,y)$. This follows easily from the continuity of $G_0(x,y)$. Also, for $\l\geq1$, we may bound 
\begin{align}
&\big| \big< \dl_x , (R_0 q)^\l R_0\dl_y \big> \big| \nonumber \\
& = \bigg| \int G_0(x,x_1) \bigg( \prod_{j=1}^{\l-1} q(x_j) G_0(x_j, x_{j+1}) \bigg) q (x_\l) G_0(x_\l, y) \, dx_1 \cdots dx_\l \bigg| \nonumber\\
& \leq \frac{1}{2\kappa} \bigg(\frac{\|q\|_{L^\infty}}{2\kappa} \bigg)^\l \int e^{-\kappa|x-x_1| - \kappa|x_1-x_2| - \cdots - \kappa|x_{\l-1}-x_\l| - \kappa|x_\l - y| } \, dx_1 \cdots dx_\l \nonumber\\ 
& \leq \frac{1}{2\kappa} \bigg(\frac{\|q\|_{L^\infty}}{2\kappa} \bigg)^\l e^{-\frac\kappa2|x-y|}\int e^{-\frac\kappa2|x-x_1| - \frac\kappa2 |x_1-x_2| - \cdots - \frac\kappa2|x_{\l-1}-x_\l| - \frac\kappa2|x_\l - y| } \, dx_1 \cdots dx_\l \nonumber\\
& \leq \frac{1}{4\kappa} \bigg(\frac{2 \|q\|_{L^\infty}}{\kappa^2} \bigg)^\l e^{-\frac\kappa2|x-y|} ,\label{g-term-bound}
\end{align}
where in the last step we used the Cauchy--Schwarz inequality in the $x_1$ variable and integrated in the remaining variables. Therefore, the series \eqref{greens} converges absolutely and uniformly in $(x,y)$ and
\begin{align*}
\bigl|G(x,y) -G_0(x,y) \bigr| & \leq  \tfrac{1}{4\kappa} e^{-\frac\kappa2|x-y|} 
\end{align*}
whenever $\kappa^2\geq 4\|q\|_{L^\infty}$.  This proves the continuity of $G(x,y)$, \eqref{G-upper}, and \eqref{g-bdd}.

The $x\leftrightarrow y$ symmetry of $G$ is inherited directly from the symmetry of the individual terms in the series, which in turn follows from the corresponding symmetry of $G_0$.   

We turn now to the behavior of the Green's function under the conditions \eqref{qn to q}.  Let us first observe that for any $\ell\geq 1$,
$$
q_n(x_1) q_n(x_2) q_n(x_3) \cdots q_n(x_\ell)  \weakstar q(x_1) q(x_2) q(x_3) \cdots q(x_\ell) \quad\text{in $L^\infty(\R^\ell)$.}
$$
This is a direct consequence of the fact that any function in $L^1(\R^\ell)$ may be approximated by a finite linear combination of indicator functions of rectangles.

Looking to the expression \eqref{Neumann term}, we see that the weak-$*$ convergence just observed guarantees that  $\big< \dl_x , (R_0 q_n)^\l R_0\dl_y \big>$ converges to the corresponding term for $q$ for each fixed $x,y\in \R$.   Pointwise convergence of the full series then follows from this and the bound \eqref{g-term-bound}, which controls the tail of the series.
\end{proof}

Using the maximum principle, one can obtain sharp upper and lower pointwise bounds on the Green's function in terms of the $L^\infty$ norm of $q$. However, these bounds are more cumbersome than  \eqref{G-upper} and \eqref{g-bdd} and provide no advantage for the arguments we will be presenting.

\begin{lemma}\label{L:g deriv}
Let $q\in L^\infty(\R)$ and $\kappa^2 \geq 4 \|q\|_{L^\infty}$. Then $G(x,y)$ is an absolutely continuous function of $x$; moreover, for $x\neq y$,
\begin{equation}\label{G-prime-bdd}
\Big|\tfrac{d}{dx} G(x,y)\Big| \leq \tfrac{3}{4} e^{-\frac\kappa2|x-y|}.
\end{equation}
The diagonal Green's function $g$ is differentiable and 
\begin{align}\label{g-prime}
|g'(x)| \leq \tfrac12.
\end{align}
Finally, if $q_n\in L^\infty(\R)$ satisfy \eqref{qn to q}, then 
\begin{align}\label{Gn-prime-conv}
\tfrac{d}{dx} G(x,y;q_n) \to \tfrac{d}{dx}G(x,y;q) \, \text{ pointwise a.e.} \quad \text{ and } \quad g'_n \to g' \text{ pointwise as } n\to \infty.
\end{align}
Here and below, $g_n(x) = g(x;q_n)$.
\end{lemma}

\begin{proof}
Although $G_0(x,y)$ is only classically differentiable where $x\neq y$, it is Lipschitz and so absolutely continuous with distributional derivative
\begin{equation}\label{G_0 prime}
\tfrac{d}{dx}  G_0(x,y) = - \tfrac12 \sgn(x-y) e^{-\kappa|x-y|}  .
\end{equation}
By comparison, $G_0(x,x)\equiv \frac1{2\kappa}$ and so differentiable in the classical sense.

One easily sees (via dominated convergence) that for each $\ell\geq 1$, the term \eqref{Neumann term} is differentiable (in the classical sense) with respect to $x$.  Moreover,
\begin{align*}
\tfrac{d}{dx}\big< \dl_x , (R_0 q)^\l R_0\dl_y \big>
	= \int \tfrac{\partial G_0}{\partial x}(x,x_1) \bigg( \prod_{j=1}^{\l-1} q(x_j) G_0(x_j, x_{j+1}) \bigg) q (x_\l) G_0(x_\l, y) \, dx_1 \cdots dx_\l .
\end{align*}
Mimicking \eqref{g-term-bound}, we find that
\begin{align}\label{11:11}
\Big| \tfrac{d}{dx} \big< \dl_x , (R_0 q)^\l R_0\dl_y \big> \Big| \leq \frac{1}{4} \bigg(\frac{2 \|q\|_{L^\infty}}{\kappa^2} \bigg)^\l e^{-\frac\kappa2|x-y|} .
\end{align}

This bound can be summed in $\ell$, showing that $G(x,y)-G_0(x,y)$ is everywhere classically differentiable and that
$$
\Big|\tfrac{d}{dx} \bigl[ G(x,y) - G_0(x,y) \bigr]\Big| \leq \tfrac{1}{4} e^{-\frac\kappa2|x-y|} \qtq{and} \Big|\tfrac{d}{dx} \bigl[ G(x,x) - G_0(x,x) \bigr]\Big| \leq \tfrac{1}{2} 
$$
whenever $\kappa^2\geq 4\|q\|_{L^\infty}$.  The claims \eqref{G-prime-bdd} and \eqref{g-prime} follow from this and our observations about $G_0$.

The claims \eqref{Gn-prime-conv} follow easily via the model laid out in the proof of Proposition~\ref{prop:Green}: one observes convergence for each individual term in the series and then exploits \eqref{11:11} in order to sum.
\end{proof}

\section{Uniqueness}\label{sec:3}
In this section, we prove Theorem~\ref{th:unique}.  Our argument will ultimately reduce to an application of the Gronwall inequality based on a subtle choice of distance function.  Our first task is to derive the integral identity to which we will apply the Gronwall inequality.  For smooth solutions, this is a direct but lengthy computation.  As our solutions are merely bounded, the derivation requires an approximation argument and hence the consideration of a forced KdV equation. 

By analogy with Definition~\ref{def:solution}, we define solutions to the forced KdV equation
\begin{equation}\label{forcing}
\tfrac{d}{dt} q = - q''' + 6 qq' + F
\end{equation}
with forcing $F\in L^\infty(\R^2)$ to be any weak-$\ast$ continuous distributional solution.

\begin{proposition}\label{prop:identities}
Let $q_1,q_2 \in L^\infty((-T,T)\times\R)$ be solutions to \eqref{forcing} with $q_1(0) = q_2(0)$ and smooth and bounded spacetime forcing terms $ F_1, F_2$, respectively.
Suppose $\kappa^2 \geq 4 \max\{\|q_{1}\|_{L_{t,x}^\infty}, \|q_{2}\|_{L_{t,x}^\infty}\}$ and let $g_1(t,x)=g(x;\kappa, q_1(t))$ and $g_2(t,x)=g(x;\kappa, q_2(t))$. Then for all $-T< t_0<T$ and every $\psi \in \mathcal S(\R)$, we have that
\begin{align}
\bigg[ \int_\R \frac{(g_1-g_2)^2}{2g_1g_2}& \psi \, dx \bigg] (t_0)\nonumber\\
& = \int_0^{t_0} \int_\R \frac{(g_1-g_2)^2}{2g_1g_2} \psi \bigg\{ -\frac{\psi'''}{2\psi} + \frac32\frac{\psi''}{\psi}  A_2 - \frac32  \frac{\psi'}{\psi} A_1 + \frac32 A_0 \bigg\} \, dx \, dt  \nonumber\\
&\quad + \frac32 \int_0^{t_0} \int_\R \psi' \, \frac{g_1-g_2}{g_1g_2} \bigg[ 2q_1 g_1 -  2q_2g_2 + \frac{(g'_1)^2}{2g_1} - \frac{(g'_2)^2}{2g_2} \bigg] \,dx \, dt \nonumber\\
&\quad - \int_0^{t_0} \int_\R \psi \,  \frac{g_1^2 - g_2^2}{2g_1g_2}\bigg(\frac{\widetilde{F}_1}{g_1} - \frac{\widetilde{F}_2}{g_2} \bigg) dx\, dt, \label{gronwall}
\end{align}
where the spacetime functions $\widetilde{F}_j, A_0,A_1,A_2$ are given by
\begin{align*}
&\widetilde{F}_j(t,x)  := \int G(x,y; \kappa, q_j(t)) F_j(t,y) G(y,x;\kappa,q_j(t)) \, dy, \\
&A_2 : = - \bigg(\frac{g_1'}{g_2} + \frac{g_2'}{g_1} \bigg)+ 2 \bigg(\frac{g_1'}{g_1} + \frac{g_2'}{g_2} \bigg) , \\
&A_1: = \frac52 \bigg( \frac{g'_1}{g_1} + \frac{g'_2}{g_2}\bigg)^2 - 7\frac{g'_1g'_2}{g_1g_2} - 12 \kappa^2 + \frac32 \frac{(g_1'-g_2')^2}{g_1g_2} + \bigg( \frac12 - 2 g'_1 g'_2\bigg) \bigg( \frac{1}{g_1} - \frac{1}{g_2} \bigg)^2 \\
&\qquad \,+ 2 q_1 \frac{g_1}{g_2} + 2 q_2 \frac{g_2}{g_1} + 2\kappa^2 \bigg( \frac{g_1}{g_2} + \frac{g_2}{g_1} \bigg) - 2 (q_1 + q_2) , \\
&A_0: = - \bigg( \frac{g_1'}{g_1^3} + \frac{g_2'}{g_2^3} \bigg) - \bigg( 2 \kappa^2 - \frac{g_1'g_2'}{2g_1g_2}\bigg) \bigg( \frac{g_1'}{g_1} + \frac{g_2'}{g_2} \bigg) - 2\bigg( \frac{q_1 g_2' }{g_2} +\frac{q_2 g_1' }{g_1} \bigg)\\
&\qquad \,+ \frac{1}{2g_1g_2}\bigg( \frac{g_1' }{g_2} + \frac{g_2'}{g_1} \bigg) .
\end{align*}
In addition, for every fixed $-T<t_0<T$, 
\begin{align}
\int_\R (q_1-q_2)&(t_0) \psi \, dx\nonumber\\
& = \frac14 \int_\R  (g_1-g_2) \psi \bigg\{ \frac{\psi''}{\psi} \bigg(\frac{1}{g_1} + \frac{1}{g_2} \bigg) + \frac12\frac{ \psi'}{\psi}  \bigg[3 \bigg( \frac{1}{g_1} + \frac{1}{g_2} \bigg)' + \frac{g'_1+g'_2}{g_1g_2} \bigg] \nonumber \\
&\quad- \frac12\bigg(\frac{1}{g_1} + \frac{1}{g_2} \bigg) \bigg[ 2(q_1+q_2) + 4\kappa^2 - \frac12\bigg(\frac{1}{g_1} + \frac{1}{g_2}\bigg)^2 + \frac{3}{g_1g_2} \bigg] \nonumber\\
& \quad+ \frac{3}{4} \bigg(\frac{(g'_1)^2}{g_1^3} + \frac{(g'_2)^2}{g_2^3} \bigg) - \frac{1}{4g_1g_2} \bigg(\frac{(g'_1)^2}{g_1} + \frac{(g'_2)^2}{g_2} \bigg)\bigg\} dx.\label{q-diff}
\end{align}
\end{proposition}

We will first establish Proposition~\ref{prop:identities} for Schwartz solutions and then employ an approximation argument to treat the case of merely bounded solutions. We start by recalling some known results regarding the diagonal Green's function; see \cite[Lemma~2.14]{KMV} and \cite[Lemma~2.6]{KV19}:

\begin{lemma}
Fix $q \in \mathcal{S}(\R)$ and $\kappa^2 \geq 4 \|q\|_{L^\infty}$. Then $g(x) -\frac{1}{2\kappa}\in \mathcal{S}(\R)$ and the following equations hold:
\begin{align}
q  &= \Big[ \frac{g'}{2g} \Big]' + \Big[ \frac{g'}{2g} \Big]^2 + \frac{1}{4 g^2 } - \kappa^2,\label{q-equation}\\
g'' &= 2 [q+\kappa^2]g + \frac{(g')^2}{2g} - \frac{1}{2g}, \label{g-double-prime} \\
g''' &= 2[q g]' + 2 q g' + 4 \kappa^2 g' \label{elliptic}.
\end{align}
Also, if $f \in \mathcal{S}(\R)$, then
\begin{equation}\label{G-identity}
\int G(x,y;q)[-f''' + 2 q f' + 2 (q f)' + 4 \kappa^2 f'](y) G(y,x;q) \, dy = 2  f'(x)g(x) - 2 f(x) g'(x). 
\end{equation}
\end{lemma}

Next we extend known formulas for the dynamics of the diagonal Green's function to the case of forced KdV.

\begin{lemma}
Let $q(t)\in \mathcal{S}(\R)$ be a Schwartz solution of the forced KdV equation \eqref{forcing} with smooth spacetime forcing $F\in L^\infty(\R^2)$ and fix $\kappa^2 \geq 4\|q\|_{L^\infty_{t,x}}$. Then the following hold:
\begin{align}
\frac{d}{dt} g 
&= \Big\{ - g'' + \frac{3(g')^2}{2g} - \frac{3}{2g} - 6 \kappa^2 g + 6 \kappa\Big\}' - \widetilde{F}, \label{g-dt}\\
\frac{d}{dt} \frac{1}{2 g} 
& = \Big\{ - \Big(\frac{1}{2g}\Big)'' + \frac{3(g')^2}{4g^3} + \frac{1}{4 g^3} - \frac{3\kappa^2}{g} + 4 \kappa^3 \Big\}' + \frac{ \widetilde{F}}{2g^2}, \label{one-g-dt}
\end{align}
where the modified forcing $\widetilde F$ is defined by
$$\widetilde{F}(t,x) := \int G(x,y; \kappa, q(t)) F(t,y) G(y,x;\kappa,q(t)) \, dy$$
and belongs to $L^\infty(\R^2)$.
\end{lemma}

\begin{proof}
The convergence of the integral defining $\widetilde F$ and the fact that $\widetilde F\in L^\infty(\R^2)$ follow immediately from \eqref{G-upper}.

From the resolvent identity
\begin{align*}
R(\kappa;q(t+h)) - R(\kappa;q(t)) = R(\kappa;q( t+h )) \big[ q(t) - q(t +h)\big] R(\kappa;q(t))
\end{align*}
and \eqref{G-identity}, we get
\begin{align*}
\frac{d}{dt} g(x; q(t)) & = - \int G(x,y;q(t)) \frac{d}{dt}q(t,y) G(y,x;q(t)) \, dy \\
& = - \int G(x,y;q(t))\big[ - q''' + 6 q q' + F \big](t,y) G(y,x;q(t)) \, dy \\
& = 2 q(x) g'(x;q(t)) - 2 q'(x) g(x;q(t)) - 4 \kappa^2 g'(x;q(t)) - \widetilde{F}(t,x).
\end{align*}
The representation \eqref{g-dt} follows from the above and the identities \eqref{g-double-prime} and \eqref{elliptic}.
Lastly, \eqref{one-g-dt} follows from \eqref{g-dt} and the chain rule. 
\end{proof}

\begin{proof}[Proof of Proposition~\ref{prop:identities} for Schwartz solutions] For notational simplicity, we forgo writing the time and space dependence for the Schwartz solutions $q_1,q_2$, their Green's functions $g_1,g_2$, and their forcing terms. 

In order to establish \eqref{gronwall}, we must compute the evolution of the quantity
$$
\frac{(g_1 - g_2)^2}{2 g_1 g_2} = \Big( \frac{g_1}{2 g_2}  + \frac{g_2}{2g_1} - 1\Big).
$$
From \eqref{g-dt} and \eqref{one-g-dt}, we have
\begin{multline*}
\frac{d}{dt} \Big( \frac{ g_1 }{2 g_2 } - \frac12 \Big)  = \frac{1}{2 g_2  } \Big\{ -  g_1''  +  \frac{3( g_1' )^2}{  2g_1 } - \frac{3}{ 2g_1 } - 6\kappa^2  g_1   \Big\}' - \frac{\widetilde{F_1}}{2g_2}  \\
+  g_1   \Big\{ - \Big[\frac{1}{2  g_2 }\Big]'' +  \frac{3(g_2' )^2}{4g_2^3} + \frac{1}{4g_2^3} - \frac{3\kappa^2}{ g_2 }  \Big\}' + \frac{g_1 \widetilde{F_2}}{2g_2^2} .
\end{multline*}
After considerable rearrangement this yields
\begin{align}
\frac{d}{dt} &\frac{(g_1-g_2)^2}{2g_1g_2} \notag\\
&= \frac12 \bigg( \frac{(g_1-g_2)^2}{2g_1g_2} \bigg)''' +\frac32 \bigg\{ \frac{(g_1-g_2)^2}{2g_1g_2} \bigg[ -\frac{g_1'}{g_2} - \frac{g_2'}{g_1} + 2\frac{g_1'}{g_1} + 2 \frac{g_2'}{g_2}\bigg]\bigg\}'' \notag\\
&\quad+ \frac32 \bigg\{ \frac{(g_1-g_2)^2}{2g_1g_2} \bigg[ 3 \frac{(g_1')^2}{g_1^2} +  3\frac{(g_2')^2}{g_2^2}- 4\kappa^2 + \frac{1}{g_1g_2} + \frac{g_1'g_2'}{g_1g_2} + \frac{(g_1'-g_2')^2}{g_1g_2}\notag\\
&  \phantom{xxxxxxxxxxxx}-2 g_1'g_2' \bigg( \frac{1}{g_1^2} + \frac{1}{g_2^2} \bigg) + \frac{g_1''}{g_2} + \frac{g_2''}{g_1} - \frac{g_1''}{g_1} - \frac{g_2''}{g_2} \bigg] \bigg\}' \notag\\
&\quad - \frac32 \frac{(g_1-g_2)^2}{2g_1g_2} \bigg\{\frac{g_1'}{g_1^3} + \frac{g_2'}{g_2^3} - \frac{g_1'g_2'}{g_1g_2} \bigg[\frac{g_1'}{g_1} + \frac{g_2'}{g_2}\bigg] + \frac{g_1''g_2' + g_1' g_2''}{g_1g_2} \bigg\}\notag\\
&\quad - \frac32 \bigg\{\frac{(g_1-g_2)(g_1''-g_2'')}{g_1g_2} \bigg\}'
- \frac{g_1^2-g_2^2}{2g_1g_2} \bigg( \frac{\widetilde{F_1}}{g_1} - \frac{\widetilde{F_2}}{g_2}\bigg). \label{micro gem}
\end{align}

 Using \eqref{g-double-prime} to eliminate $g''_1$ and $g''_2$, we obtain
\begin{align*}
\frac{d}{dt} &\frac{(g_1-g_2)^2}{2g_1g_2} \\
& = \frac12 \bigg( \frac{(g_1-g_2)^2}{2g_1g_2} \bigg)''' +\frac32 \bigg( \frac{(g_1-g_2)^2}{2g_1g_2} A_2 \bigg)'' + \frac32 \bigg( \frac{(g_1-g_2)^2}{2g_1g_2}  A_1 \bigg)' + \frac32 \frac{(g_1-g_2)^2}{2g_1g_2} A_0 \\
&\quad - \frac32 \bigg\{ \frac{g_1-g_2}{g_1g_2} \bigg[ 2q_1g_1 - 2q_2 g_2 + \frac{(g'_1)^2}{2g_1} - \frac{(g'_2)^2}{2g_1} \bigg]\bigg\}'  - \frac{g_1^2-g_2^2}{2g_1g_2} \bigg( \frac{\widetilde{F}_1}{g_1} - \frac{\widetilde{F}_2}{g_2}\bigg)
\end{align*}
with $A_0,A_1,A_2$ as given in Proposition~\ref{prop:identities}. The equation \eqref{gronwall} follows from integrating against a test function $\psi\in C_c^\infty(\R)$ and then integrating by parts.

We now prove \eqref{q-diff}. From \eqref{q-equation} we have
\begin{align*}
q_1-q_2 & = \bigg(\frac{g'_1}{2g_1} - \frac{g'_2}{2g_2}\bigg)' + \bigg( \frac{g'_1}{2g_1}\bigg)^2  -  \bigg( \frac{g'_2}{2g_2}\bigg)^2 + \frac{1}{4g_1^2} - \frac{1}{4g_2^2}, 
\end{align*}
which we rewrite as follows
\begin{multline*}
q_1-q_2 = \frac14 \bigg\{ (g_1-g_2)\bigg(\frac{1}{g_1} + \frac{1}{g_2} \bigg)\bigg\}'' + \frac18 \bigg\{ (g_1-g_2)\bigg[- 3 \bigg(\frac{1}{g_1} + \frac{1}{g_2} \bigg)' - \frac{g'_1+g'_2}{g_1g_2} \bigg] \bigg\}' \\
+\frac18 (g_1-g_2) \bigg\{ -\bigg( \frac{1}{g_1} + \frac{1}{g_2} \bigg) \bigg( \frac{g''_1}{g_1} + \frac{g''_2}{g_2} + \frac{2}{g_1g_2} \bigg)  + 2 \bigg( \frac{(g'_1)^2}{g_1^3} + \frac{(g'_2)^2}{g_2^3}\bigg)\bigg\}.
\end{multline*}
Using \eqref{g-double-prime} to eliminate $g''_1$ and $g''_2$, we find that
\begin{align*}
q_1-q_2 = \frac14 \bigg\{ (g_1-g_2)\bigg(\frac{1}{g_1} + \frac{1}{g_2} \bigg)\bigg\}'' + \frac18 \bigg\{ (g_1-g_2)\bigg[- 3 \bigg(\frac{1}{g_1} + \frac{1}{g_2} \bigg)' - \frac{g'_1+g'_2}{g_1g_2} \bigg]\bigg\}' \\
+\frac18 (g_1-g_2) \bigg\{ -\bigg( \frac{1}{g_1} + \frac{1}{g_2} \bigg) \bigg[2(q_1+q_2) + 4\kappa^2  - \frac12\bigg(\frac{1}{g_1} + \frac{1}{g_2}\bigg)^2 + \frac{3}{g_1g_2} \bigg] \\
+ \frac32 \bigg( \frac{(g'_1)^2}{g_1^3} + \frac{(g'_2)^2}{g_2^3}\bigg) -\frac{1}{2g_1g_2} \bigg(\frac{(g'_1)^2}{g_1} + \frac{(g'_2)^2}{g_2} \bigg)\bigg\}, 
\end{align*}
from which \eqref{q-diff} follows by integrating against the Schwartz function $\psi$ and then integrating by parts. 
\end{proof}

\begin{proof}[Proof of Proposition~\ref{prop:identities} for bounded solutions] Our argument will be to mollify the solutions $q_1,q_2 \in L^\infty((-T,T) \times \R)$ so that they are smooth functions of spacetime and of Schwartz class at each fixed time.  Inevitably, this will lead to changes in the forcing terms.  Ultimately, we will show that our mollified sequences of solutions, their diagonal Green's functions, and the forcing terms converge sufficiently well to ensure that \eqref{gronwall} and \eqref{q-diff} carry over from Schwartz solutions to the case of merely bounded solutions.

Let us now explain how we mollify a bounded solution $q$ to \eqref{forcing} with smooth and bounded forcing term $F$; the procedure is applied equally to $q_1$ and $q_2$.  Given a non-negative $\phi \in C_c^\infty(\R)$ with $\supp \phi \subset[-1,1]$ and $\int \phi(x) \, dx = 1$, we define
\begin{gather*}
\phi_n(x) = n \phi(nx), \qquad \varphi_n(t,x) = \phi_n(t) \phi_n(x), \qquad \psi_n(x) = \operatorname{sech}\Big(\frac{x}{n^3}\Big), \nonumber\\
q_n(t,x) = \psi_n(x) (\varphi_n \ast_{t,x} q )(t,x) = \psi_n(x) \int_{\R\times \R} \phi_n(t-\tau) \phi_n(x-y) q(\tau,y) \, d\tau \, dy .
\end{gather*}
By direct computation, we see that $q_n$ solves the following forced KdV equation
\begin{align}\label{7:13}
\tfrac{d}{dt} q_n &=  - \psi_n (\varphi_n''' \ast_{t,x} q )+ 3 \psi_n (\varphi_n' \ast_{t,x} q^2 )+ \psi_n (\varphi_n \ast_{t,x} F)\notag \\
 & = -q'''_n + 6 q_n q'_n + E_{1,n} + E_{2,n} + 3E_{3,n}'
\end{align}
where 
\begin{align*}
E_{1,n} & = \psi_n (\varphi_n \ast_{t,x} F),\\
E_{2,n} & = \psi_n'''(\varphi_n \ast_{t,x}q) + 3 \psi_n'' (\varphi_n' \ast_{t,x} q) + 3 \psi'_n \big[ \varphi_n'' \ast_{t,x} q - \varphi_n \ast_{t,x} q^2 \big],  \\
E_{3,n} & = \psi_n (\varphi_n \ast_{t,x} q^2) - q_n^2.
\end{align*}

Regarding the convergence of $q_n$ to $q$, it will suffice for us show that
\begin{align}\label{bullet 1}
\| q_n \|_{L^\infty_{t,x}} \leq \| q\|_{L^\infty_{t,x}} \qtq{and} q_n(t) \weakstar q (t) \text{ in $L^\infty(\R)$ for each $t\in(-T,T)$}.
\end{align}
The former claim is elementary; it guarantees that the diagonal Green's functions of $q_n$ are defined for the same range of $\kappa$ as we would employ for $q$.  It will also allow us to apply the dominated convergence theorem later because
\begin{align}\label{bullet 1'}
q_n(t,x) \to q (t,x) \quad \text{pointwise a.e. on $(-T,T)\times\R$ as $n\to \infty$}.
\end{align}

The second claim in \eqref{bullet 1} warrants a little explanation.  Given $t\in (-T,T)$ and $f\in L^1(\R)$, we have
\begin{align*}
\bigl\langle q_n(t) , f \bigr\rangle = \bigl\langle [\phi_n *_t q](t)  ,  \phi_n *_x(\psi_n f) \bigr\rangle .
\end{align*}
As the solution $q$ is weak-$*$ continuous, it follows that $[\phi_n *_t q](t) \weakstar q(t)$ in $L^\infty(\R)$.  The claim then follows from this and the fact that $ \phi_n *_x(\psi_n f)\to f$ in $L^1(\R)$.

By Proposition~\ref{prop:Green} and Lemma~\ref{L:g deriv}, \eqref{bullet 1} guarantees the convergence of the diagonal Green's functions $g_n(t,x) = g_n(x;\kappa, q_n(t))$ and their derivatives.  Specifically,  for $t\in (-T,T)$ we have
\begin{align}\label{bullet 2}
g_n (t) \to g(t) \qtq{and} g_n' (t) \to g'(t) \quad \text{pointwise as $n\to \infty$}.
\end{align}
Moreover, by \eqref{g-bdd} and \eqref{g-prime}, we have the uniform bounds
\begin{align}\label{bullet 3}
\tfrac1{4\kappa}\leq g_n(t,x) \leq  \tfrac3{4\kappa}\qtq{and} |g_n' (t,x)|\leq \tfrac 12.
\end{align}

As $F$ is bounded and continuous, we have that
\begin{align}\label{e1}
\|E_{1,n}\|_{L^\infty_{t,x}} \leq \|F\|_{L^\infty_{t,x}} \quad \text{and}\quad  E_{1,n}(t,x) \to F(t,x) \text{ pointwise as } n \to \infty.
\end{align}
Regarding the error term $E_{2,n}$, we note that for $\l=0,1,2$,
$$\|\varphi^{(\l)}_n \ast_{t,x} q \|_{L^\infty_{t,x}} \les n^{\l} \|q\|_{L^\infty_{t,x}}, \qquad \| \varphi_n \ast_{t,x} q^2\|_{L^\infty_{t,x}} \les \|q\|_{L^\infty_{t,x}}^2,$$
and that $\|\psi^{(\ell)}_n \|_{L^\infty} \les n^{-3\ell}$ for $\ell=1,2,3$. From this, it follows easily that 
\begin{align}\label{e2}
\|E_{2,n}\|_{L^\infty_{t,x}} \les \|q\|_{L^\infty_{t,x}} + \|q\|_{L^\infty_{t,x}}^2 \qtq{and} E_{2,n}(t,x) \to 0 \text{ pointwise as $n\to \infty$}.
\end{align}
The same arguments also yield that 
\begin{align}\label{e3}
\|E_{3,n}\|_{L^\infty_{t,x}} \les \|q\|_{L^\infty_{t,x}}^2  \qtq{and}  E_{3,n}(t,x) \to 0 \text{ pointwise as $n\to \infty$} .
\end{align}
We caution the reader that this last error term appears differentiated in \eqref{7:13}.

The error terms $E_{1,n}$, $E_{2,n}$, and $E_{3,n}$ do not contribute directly to the identities \eqref{gronwall} and \eqref{q-diff}, but only through 
\begin{align*}
\widetilde{F}_n(t,x) = \int_\R G_n(x,y) \big[  E_{1,n} + E_{2,n} + 3E_{3,n}'\big](t,y) G_n(y,x) \, dy.
\end{align*}
Here $G_n(x,y)=G(x,y;\kappa, q_n(t))$.  Our next lemma shows that $\widetilde{F}_n(t,x)$ converges pointwise to
\begin{align*}
\widetilde{F}(t,x) = \int_\R G(x,y) F(t,y) G(y,x) \, dy.
\end{align*}

\begin{lemma}\label{lemma:error-forcing}
The functions $\widetilde F_n$ are bounded uniformly in $n,t,x$ and 
\begin{align*}
\widetilde{F}_n(t,x) \to \widetilde F(t,x) \quad \text{pointwise as $n\to\infty$}. 
\end{align*}
\end{lemma}

\begin{proof}
To handle the derivative appearing on $E_{3,n}$, we integrate by parts:
\begin{align}\label{6:39}
\int G_n(x,y) E_{3,n}'(y) G_n(y,x) \, dy 
&= -2 \int  \Big[\frac{\partial}{\partial y} G_n(y,x) \Big] E_{3,n}(y) G_n(x,y) \, dy. 
\end{align}
Here, we implicitly used the $x\leftrightarrow y$ symmetry of the Green's function and the fact the $E_{3,n}$ converges to zero at spatial infinity due to the presence of $\psi_n$ factors.

From the bounds \eqref{G-upper} and \eqref{G-prime-bdd}, we have 
\begin{align}\label{6:46}
|G_n(x,y) | \leq \tfrac{3}{4\kappa} e^{-\frac\kappa2|x-y|} \qtq{and} \Big|\tfrac{d}{dx} G_n(x,y)\Big| \leq \tfrac{3}{4} e^{-\frac\kappa2|x-y|}.
\end{align}
These bounds hold for all $n\in \NB$ because of \eqref{bullet 1}.

The uniform boundedness of $\widetilde F_n$ follows from \eqref{6:39}, \eqref{6:46}, and the uniform boundedness of the error terms $E_{1,n}$, $E_{2,n}$, and $E_{3,n}$ observed in \eqref{e1}, \eqref{e2}, and \eqref{e3}, respectively.

The fact that $\widetilde F_n$  converges pointwise to $\widetilde F$ follows from the dominated convergence theorem, \eqref{e1}, \eqref{e2}, \eqref{e3}, and  \eqref{6:46}.
\end{proof}

Having gathered all the necessary convergence results, we are now ready to complete the proof of Proposition~\ref{prop:identities}.  We do this by sending $n\to \infty$ in the identities \eqref{gronwall} and \eqref{q-diff} satisfied by the Schwartz solutions $q_n$ to the forced KdV \eqref{7:13}.  For all terms in \eqref{gronwall}, we may apply the dominated convergence theorem using \eqref{bullet 1'}, \eqref{bullet 2}, \eqref{bullet 3}, and Lemma~\ref{lemma:error-forcing}.  

This argument also applies to the terms appearing on the right-hand side of \eqref{q-diff}. It does not apply to the left-hand side of \eqref{q-diff} because we are not guaranteed that $q_{j,n}(t)$ converges pointwise a.e. to $q_j(t)$ for \emph{every} $t$.  This is remedied by \eqref{bullet 1}.
\end{proof}

We are now ready to prove our uniqueness result.

\begin{proof}[Proof of Theorem~\ref{th:unique}] Due to time-translation invariance, it suffices to prove uniqueness of solutions on intervals of the form $(-T,T)$.

Consider two solutions $q_1,q_2:(-T,T)\to L^\infty(\R)$ to \eqref{kdv} in the sense of Definition~\ref{def:solution} with the same initial data $q_1(0) = q_2(0)$.  Due to the time-reversal symmetry $q(t,x)\mapsto q(-t,-x)$ of the equation, it suffices to show uniqueness forward in time. 

Let us fix  $\kappa^2 \geq 4 \max\{\|q_1\|_{L^\infty_{t,x}}, \|q_2\|_{L^\infty_{t,x}}\}$. We first prove that the corresponding diagonal Green's functions $g_1$ and $g_2$ agree at any fixed time $t_0\in (0,T)$.  We will then deduce the equality of $q_1(t_0)$ and $q_2(t_0)$. 

From Proposition~\ref{prop:identities}, the identity \eqref{gronwall} holds for $q_1,q_2$ with $F_1=F_2 =\widetilde{F}_1 =\widetilde{F}_2= 0$, and $\psi_R (x) : = \operatorname{sech}(\frac{x}{R})$ with $R\geq 1$.  Moreover, from \eqref{g-bdd} and \eqref{g-prime}, we have the following estimates for the quantities appearing in \eqref{gronwall}:
\begin{gather*}
\big\| \tfrac{\psi_R'''}{\psi_R} \big\|_{L^\infty_{t,x}}  \les \tfrac{1}{R^3}, \quad \big\| \tfrac{\psi_R''}{\psi_R} A_2 \big\|_{L^\infty_{t,x}} \les \tfrac{\kappa}{R^2}, \quad
\big\| \tfrac{\psi'_R}{\psi_R} A_1 \big\|_{L^\infty_{t,x}}  \les \tfrac{\kappa^2 + \|q_1\|_{L^\infty_{t,x}} + \|q_2\|_{L^\infty_{t,x}}}{R} \les \tfrac{\kappa^2}{R}, \\
\|A_0\|_{L^\infty_{t,x}}  \les \kappa^3 + \kappa\big[\|q_1\|_{L^\infty_{t,x}} + \|q_2\|_{L^\infty_{t,x}}\big] \les \kappa^3.
\end{gather*}

Using $|\psi_R'|\lesssim R^{-1} \psi_R$ and the Cauchy--Schwarz inequality, we may bound
\begin{align*}
& \bigg| \int_0^{t_0} \int_\R \psi'_R \frac{g_1 - g_2}{g_1g_2}\bigg[2 q_1 g_1-2 q_2 g_2 + \frac{(g'_1)^2}{2g_1} - \frac{(g'_2)^2}{2g_2}\bigg] \, dx \, dt \bigg| \\
& \les \frac{1}{R } \int_0^{t_0} \int_\R \psi_R  \frac{1}{g_1g_2} \bigg[ \frac{1}{2\eps}(g_1-g_2)^2 + \frac\eps 2 \bigg( 2q_1 g_1 - 2 q_2 g_2 + \frac{(g'_1)^2}{2g_1} - \frac{(g'_2)^2}{2g_2} \bigg)^2 \bigg] \, dx \, dt \\
& \les\frac{1}{\eps R } \int_0^{t_0} \int_\R \psi_R \frac{(g_1-g_2)^2}{2g_1g_2} \, dx \, dt + \frac{\eps}{2R } \bigl[\|q_1\|_{L^\infty_{T,x}} + \|q_2\|_{L^\infty_{T,x}} + \kappa^2\bigr]^2 \int_0^{t_0} \int \psi_R\, dx \,dt \\
& \les  \frac{1}{\eps R } \int_0^{t_0} \int_\R \psi_R \frac{(g_1-g_2)^2}{2g_1g_2} \, dx \, dt + \eps\kappa^4 T, 
\end{align*}
for any choice of $\eps>0$.

Combining the estimates above, we obtain
\begin{align*}
\int_\R \psi_R\frac{(g_1-g_2)^2}{2g_1g_2}(t_0)  \,dx & \les \bigg[ \frac{1}{R^3} + \frac{\kappa}{R^2} + \frac{\kappa^2}{R} + \kappa^3 + \frac{1}{\eps R} \bigg] \int_0^{t_0} \int_\R \psi_R \frac{(g_1-g_2)^2}{2g_1g_2} \, dx \, dt\\
&\quad +\eps \kappa^4T.
\end{align*}
Choosing $\eps = R^{-\frac12}$, recalling that $R\geq 1$, and applying Gronwall's inequality, we conclude that
\begin{align*}
 \int_\R \psi_R\frac{(g_1-g_2)^2}{2g_1g_2}(t_0)\, dx  & \les \frac{\kappa^4T}{\sqrt{R}} e^{CT (1+ \kappa^3)},
 \end{align*}
for some constant $C>0$ independent of $R$ and $\kappa$. 

As $R_1 \leq R_2$ implies $\psi_{R_1}(x) \leq \psi_{R_2}(x)$ for all $x\in \R$, we deduce that
\begin{align*}
 \int_\R \psi_{R}\frac{(g_1-g_2)^2}{2g_1g_2}(t_0)\, dx  &\leq \lim_{\widetilde R \to \infty}   \int_\R \psi_{\widetilde R}\frac{(g_1-g_2)^2}{2g_1g_2}(t_0) \,  dx =0 , 
\end{align*}
for all $R\geq 1$.  Recalling \eqref{g-bdd}, we conclude that $g_1(t_0)\equiv g_2(t_0)$.  Finally, using the identity \eqref{q-diff} we deduce that $q_1(t_0)=q_2(t_0)$ as elements of $L^\infty(\R)$.
\end{proof}

\section{Quasiperiodic solutions to KdV}\label{sec:quasi}

In this section, we discuss the solution to \eqref{kdv} with initial data
\begin{align}\label{square}
q_0(x) = \operatorname{sq}(\al_1x) + \operatorname{sq}(\al_2x),
\end{align}
where $\operatorname{sq}$ denotes the square wave \eqref{sq wave} of period $2\pi$.  We can rewrite $q_0$ as 
\begin{align}\label{q0 hat}
q_0(x) = \sum_{\xi\in \dot{\Z}^2} \tfrac{2}{\pi i (\al\cdot \xi)} \big( \al_1 \chi_{\{\xi_1 \text{ odd}, \,\xi_2 =0\}} +\al_2 \chi_{\{\xi_1 =0 ,\, \xi_2 \text{ odd}\}} \big) e^{i (\al\cdot \xi) x }, 
\end{align}
where $\xi=(\xi_1,\xi_2) \in \dot{\Z}^2 = \Z^2 \setminus\{(0,0)\}$ and $\al = (\al_1,\al_2)$. Under the assumption that $\al$ is rationally independent, that is,  
\begin{align}\label{6:59}
\al \cdot \xi \neq 0 \qtq{for all} \xi\in\dot{\Z}^2, 
\end{align}
the data $q_0$ is not periodic, but merely quasiperiodic.  

We shall only consider parameters $\al$ satisfying a quantitative version of \eqref{6:59}, namely, the following diophantine condition: there exists $\gamma>1$ and $C_0>0$ such that 
\begin{equation}\label{diophantine}
|\al \cdot \xi| \geq C_0 |\xi|^{-\gamma} \qtq{for all} \xi\in\dot{\Z}^2.
\end{equation}

Local well-posedness of KdV for a class of quasiperiodic initial data that includes our choice \eqref{square} was proved by Tsugawa in \cite{Tsugawa12}.  Let us recall the version of the spaces he employed that are most relevant to the case considered here.

\begin{definition}[{\cite{Tsugawa12}}]
For $\theta \in \R$, the Banach space $G^\theta$ is defined by
\begin{align*}
G^\theta := \bigg\{ f(x) = \sum_{\xi\in\dot{\Z}^2} \ft{f}_\xi \, e^{i(\al\cdot \xi) x} \  \Big\vert \ \   \ft{f}: \dot{\Z}^2 \to \C, \ \|f\|_{G^\theta} < \infty\bigg\} 
\end{align*}
where 
\begin{align*}
\|f\|_{G^\theta} : = \| \ft{f} \|_{\ft{G}^\theta} = \bigg\| \frac{\jb{\xi_1}^\theta \jb{\xi_2}^\theta}{|\al\cdot \xi|^\frac12} \ft{f}_\xi \bigg\|_{\l^2_\xi(\dot{\Z}^2)}.
\end{align*}
We will also employ the $X^{s,b}$-type space defined via the norm
\begin{align*}
\|q\|_{X^{\theta,\frac12}} &:= \big\| \jb{\tau - (\al\cdot \xi)^3}^\frac12 \F_{t,x} q (\tau, \xi) \big\|_{\ft{G}^\theta L^2_\tau}.
\end{align*}
\end{definition}

We observe that the initial data in \eqref{square} satisfies
\begin{align*}
q_0 \in G^\theta \text{ for } \theta<1.
\end{align*}

Tsugawa's solutions automatically have vanishing Fourier coefficient at zero frequency; indeed, they are constructed via a contraction mapping argument in the spaces just reproduced, as well as an additional space $Y^{\theta,0}$ that we do not need to discuss here.  The vanishing of the zero Fourier coefficient may be viewed as the quasiperiodic analogue of the well-known conservation of $\int q$ and may ultimately be traced to the fact that the right-hand side of \eqref{kdv} is a complete derivative.

As with earlier works proving local well-posedness using $X^{s,b}$ technology, Tsugawa employs a truncated Duhamel formulation of the problem:
\begin{equation}\label{duhamel}
q(t) = \eta(t) e^{-t\partial_x^3} q_0 + 3 \eta(t) \int_0^t e^{-(t-s)\partial_x^3} \dx (\eta_T q^2)(s) \, ds.
\end{equation}
Here $\eta(t)$ is a fixed smooth cutoff function at unit scale, while $\eta_T(t)$ is a cutoff to a narrower time window dictated by the size of the initial data.  Evidently, fixed points of \eqref{duhamel} are solutions to \eqref{kdv} at least on the small time interval $[-T,T]$.  The big advantage of this formulation of the problem is that it allows $q(t)$ to be defined globally in time and so one may employ the standard spacetime Fourier transform.

\begin{theorem}[{\cite[Theorem 1.1]{Tsugawa12}}]\label{th:lwp}
The KdV equation \eqref{kdv} is locally well-posed in $G^\theta$ with $\theta>\frac14$ in the following sense: for each $q_0 \in G^\theta$, there exist $T>0$ and a unique solution $q \in C\big(\R; G^\theta \big) \cap X^{\theta,\frac12}$ of \eqref{duhamel}.
\end{theorem}

The remainder of this section is devoted to demonstrating nonlinear smoothing for the solution with initial data \eqref{square}.  By nonlinear smoothing, we mean that the difference between the linear and nonlinear evolutions of the initial data is smoother than the linear evolution alone.  In this section, such regularity will be expressed through enhanced decay of the Fourier coefficients.  Specifically, we will show that the Fourier coefficients of the difference belong to $\ell^1$.  By comparison, the Fourier coefficients of the initial data (and so those of its linear evolution) are merely weak-$\ell^1$; see \eqref{q0 hat}.  While this may be viewed as a minute difference, it marks a phase transition in terms of spatial continuity.

\begin{theorem}\label{prop:nonlinear-smoothing}
Let $q_0$ be as in \eqref{square} with $\al$ satisfying the diophantine condition \eqref{diophantine}.  Let $q \in C\big(\R; G^\theta \big) \cap X^{\theta,\frac12} $ be the corresponding solution of \eqref{kdv} given by Theorem~\ref{th:lwp}. If $ \max\{\frac78 , \frac\gamma2\}<\theta <1$, then
\begin{align}\label{NS}
\big\|  \ft{q}_\xi(t) - e^{it(\al\cdot \xi)^3}&\ft{q}_\xi(0)  \big\|_{L^\infty_t \l^1_\xi([-T,T]\times\dot{\Z}^2)}\notag\\
&\les (1+  \|q\|_{X^{\theta,\frac12}}) \|q\|^2_{X^{\theta,\frac12}} + (1+T) (1+\|q\|_{L^\infty_t G^{\theta}} ) \|q\|^2_{L^\infty_t G^{\theta}}.
\end{align}
\end{theorem}

Before proceeding to the proof of Theorem~\ref{prop:nonlinear-smoothing}, we recall the following standard result that allows us to handle the time cutoff appearing in \eqref{duhamel}.

\begin{lemma}\label{lm:cutoff}
Let $0\leq b <\frac12$ and fix $T>0$. Then for every $f\in H^b(\R)$, 
\begin{align*}
\| \chi_{[-T,T]}  f  \|_{H^b} \les \|f\|_{H^b}.
\end{align*}
\end{lemma}

\begin{proof}[Proof of Theorem~\ref{prop:nonlinear-smoothing}]
Using \eqref{duhamel}, for $\xi\in\dot{\Z}^2$ and $|t|\leq T$ we can write 
\begin{align}\label{smooth-q}
\ft{q}_\xi(t) - e^{it(\al\cdot \xi)^3 } \ft{q}_\xi(0)  = 3i \sum_{\xi^{(1)} + \xi^{(2)}=\xi}\int_0^t e^{i(t-s)(\al\cdot \xi)^3}  (\al \cdot \xi) \ft{q}_{\xi^{(1)}}(s)  \ft{q}_{\xi^{(2)}}(s)\,ds.
\end{align}
The sum above runs over all decompositions of $\xi$ with $\xi^{(1)},\xi^{(2)}\in \dot{\Z}^2$. 

To estimate \eqref{smooth-q}, we divide the sum into several regions.  By symmetry, it suffices to consider only the case when
\begin{align}\label{sym ass}
|\al\cdot \xi^{(1)}| \geq |\al \cdot \xi^{(2)}| \qtq{and} |\xi^{(1)}_1| \geq |\xi^{(1)}_2|,
\end{align}
where we use the notation $\xi^{(j)}=( \xi^{(j)}_1,\xi^{(j)}_2)$ with $j=1,2$.  This allows us to write
\begin{equation}\label{aux1-nonlinear}
|\al\cdot \xi| = |\al \cdot \xi^{(1)} + \al \cdot \xi^{(2)} | \les |\al \cdot \xi^{(1)} | \les |\xi^{(1)}_1|.
\end{equation}

\medskip 

\noi \textbf{Case 1: $ | \xi^{(1)}_1| \les |\xi^{(1)}_2| + |\xi_1^{(2)}| + |\xi_2^{(2)}|$.}
In view of the factorization
\begin{align}\label{factor}
(\al\cdot \xi)^3 - (\al\cdot \xi^{(1)})^3 - (\al \cdot \xi^{(2)})^3 = 3 (\al \cdot \xi ) (\al \cdot \xi^{(1)}) (\al \cdot \xi^{(2)})
\end{align}
where $\xi= \xi^{(1)} + \xi^{(2)}$, we have
\begin{align*}
2\max\bigl\{ \big| (\al\cdot \xi)^3 - \tau - (\al\cdot \xi^{(1)})^3\big|,\big |\tau - (\al\cdot \xi^{(2)})^3\big|\bigr\} &\geq 3 \big| (\al \cdot \xi) (\al \cdot \xi^{(1)}) (\al \cdot \xi^{(2)}) \big|.
\end{align*} 
We present the details in the case when
$$
\max\bigl\{ \big| (\al\cdot \xi)^3 - \tau - (\al\cdot \xi^{(1)})^3\big|,\big|\tau - (\al\cdot \xi^{(2)})^3\big|\bigr\} = \big|\tau - (\al\cdot \xi^{(2)})^3\big|.
$$
In the remaining case, one simply swaps the roles of $\ft{q}_{\xi^{(1)}}$ and $\ft{q}_{\xi^{(2)}}$ when estimating the time integral, including which term $\chi_{[0,t]}$ gets grouped with.
Here, $\chi_{[0,t]}$ denotes the sharp cutoff to the interval $[0,t]$.

To take advantage of the fact that $q\in X^{\s,\frac12}$, we rewrite RHS\eqref{smooth-q} as 
\begin{align*}
\text{RHS\eqref{smooth-q}} & = 3 e^{it (\al\cdot \xi)^3} 
\sum_{\xi^{(1)} + \xi^{(2)}=\xi } i (\al \cdot \xi) \int_\R  e^{-is(\al\cdot \xi)^3} \chi_{[0,t]}(s)\ft{q}_{\xi^{(1)}}(s)  \ft{q}_{\xi^{(2)}}(s)  \, ds \\
& = 3 \sqrt{2\pi} \ e^{it (\al\cdot \xi)^3} \!\!\!\!\!
\sum_{\xi^{(1)} + \xi^{(2)} =\xi} \!\!\!i (\al \cdot \xi) \int_\R \F_t \big( \chi_{[0,t]} \ft{q}_{\xi^{(1)}} \big) \big( (\al\cdot \xi)^3 - \tau \big) \F_t \ft{q}_{\xi^{(2)}}(\tau) \, d \tau.
\end{align*}
Using Cauchy--Schwarz, we may estimate
\begin{align*}
& \bigg| \int_\R \F_t \big( \chi_{[0,t]} \ft{q}_{\xi^{(1)}} \big) \big( (\al\cdot \xi)^3 - \tau \big) \F_t \ft{q}_{\xi^{(2)}}(\tau) \, d \tau \bigg| \\
& \les \frac{1}{|\al \cdot \xi|^\frac12 |\al \cdot \xi^{(1)}|^\frac12 |\al \cdot \xi^{(2)}|^\frac12} \int_\R \bigl| \F_t \big( \chi_{[0,t]} \ft{q}_{\xi^{(1)}} \big) \big( (\al\cdot \xi)^3 - \tau \big) \bigr| \bigl\langle\tau - (\al\cdot \xi^{(2)})^3\bigr\rangle^\frac12 \bigl| \F_t \ft{q}_{\xi^{(2)}}(\tau) \bigr| \, d\tau \\
& \les  \frac{1}{|\al \cdot \xi|^\frac12 |\al \cdot \xi^{(1)}|^\frac12 |\al \cdot \xi^{(2)}|^\frac12} \big\| \F_t \big( \chi_{[0,t]} \ft{q}_{\xi^{(1)}} \big) (\tau) \big\|_{L^2_\tau} \big\|  \bigl\langle\tau - (\al\cdot \xi^{(2)})^3\bigr\rangle^\frac12 \F_t \ft{q}_{\xi^{(2)}}(\tau) \big\|_{L^2_\tau} \\
& \les \frac{1}{|\al \cdot \xi|^\frac12 |\al \cdot \xi^{(1)}|^\frac12 |\al \cdot \xi^{(2)}|^\frac12} \prod_{j=1}^2 \bigl\|  \jb{\tau - (\al\cdot \xi^{(j)})^3}^\frac12  \F_t \ft{q}_{\xi^{(j)}}(\tau) \bigr\|_{L^2_\tau}.
\end{align*}

One more application of Cauchy--Schwarz shows that we may estimate the contribution of Case~1 to LHS\eqref{NS} by a constant multiple of
\begin{align*}
&\sum_\xi \sum_{\xi=\xi^{(1)} + \xi^{(2)}} \frac{|\al\cdot \xi|^\frac12}{|\al\cdot \xi^{(1)}|^\frac12 |\al\cdot \xi^{(2)}|^\frac12} \prod_{j=1}^2 \big\|  \jb{\tau - (\al\cdot \xi^{(j)})^3}^\frac12  \F_t \ft{q}_{\xi^{(j)}}(\tau) \big\|_{L^2_\tau} \\
& = \sum_\xi \sum_{\xi=\xi^{(1)} + \xi^{(2)}} \frac{|\al\cdot \xi|^\frac12}{\prod
\limits_{j=1}^2 \jb{\xi_1^{(j)}}^\theta \jb{\xi_2^{(j)}}^\theta} \prod_{j=1}^2 \frac{\jb{\xi^{(j)}_1}^\theta \jb{\xi^{(j)}_2}^\theta}{|\al\cdot \xi^{(j)}|^\frac12} \big\|  \jb{\tau - (\al\cdot \xi^{(j)})^3}^\frac12  \F_t \ft{q}_{\xi^{(j)}}(\tau) \big\|_{L^2_\tau} \\
& \les A^\frac12 \|q\|_{X^{\theta,\frac12}}^2, 
\end{align*}
where 
\begin{align*}
A:= \sum_\xi \sum_{\xi=\xi^{(1)} + \xi^{(2)}} \frac{|\al\cdot \xi|}{ \jb{\xi_1^{(1)}}^{2\theta} \jb{\xi_2^{(1)}}^{2\theta} \jb{\xi_1^{(2)}}^{2\theta} \jb{\xi_2^{(2)}}^{2\theta}}.
\end{align*}
Using  \eqref{aux1-nonlinear} and the description of Case~1, we may bound
$$|\al\cdot \xi|\lesssim |\xi_1^{(1)}| \lesssim \jb{\xi_1^{(1)}}^\frac12 \jb{\xi_2^{(1)}}^\frac12 \jb{\xi_1^{(2)}}^\frac12 \jb{\xi_2^{(2)}}^\frac12$$
and so
\begin{align*}
A & \les \sum_\xi \sum_{\xi=\xi^{(1)} + \xi^{(2)}} \frac{1}{ \jb{\xi_1^{(1)}}^{2\theta-\frac12} \jb{\xi_2^{(1)}}^{2\theta- \frac12} \jb{\xi_1^{(2)}}^{2\theta-\frac12} \jb{\xi_2^{(2)}}^{2\theta-\frac12}} \les 1,
\end{align*}
provided $2\theta - \frac12>1$, or equivalently, $\theta > \frac34$.

\medskip

\noi\textbf{Case 2: $ |\xi^{(1)}_1| \gg | \xi^{(1)}_2 | +  |\xi^{(2)}_1| +  |\xi^{(2)}_2|$}. We write $\mathcal R$ for the set of such decompositions of $\xi$.  To estimate this contribution to \eqref{NS}, we employ the interaction representation $u(t) = e^{t\partial_x^3}q(t)$ and a normal form transformation:
\begin{align}\label{smooth-u}
e^{-it(\al\cdot \xi)^3 }&  \ft{q}_\xi(t) - \ft{q}_\xi(0)\notag\\
&= 3\int_0^t \sum_{\mathcal R} e^{- 3 i s (\al\cdot \xi)(\al \cdot \xi^{(1)}) (\al \cdot \xi^{(2)})} i(\al \cdot \xi) \ft{u}_{\xi^{(1)}}(s)  \ft{u}_{\xi^{(2)}}(s)  \, ds\notag\\
& = 3 \int_0^t \sum_{\mathcal R} \frac{d}{ds} \Big( \frac{-e^{- 3 i s (\al\cdot \xi)(\al \cdot \xi^{(1)}) (\al \cdot \xi^{(2)})}}{ 3 i (\al\cdot \xi)(\al \cdot \xi^{(1)}) (\al \cdot \xi^{(2)})} \Big) i(\al \cdot \xi) \ft{u}_{\xi^{(1)}}(s)  \ft{u}_{\xi^{(2)}}(s)  \, ds \notag\\
& = \mathcal{B}(t) - \mathcal{B}(0) + \mathcal{N}_1(t) + \mathcal{N}_2(t),
\end{align}
where
\begin{align*}
\mathcal{B}(s) & := \sum_{\mathcal R} e^{- 3 i s (\al\cdot \xi)(\al \cdot \xi^{(1)}) (\al \cdot \xi^{(2)})} \frac{-1}{(\al \cdot \xi^{(1)}) (\al \cdot \xi^{(2)})}  \, \ft{u}_{\xi^{(1)}}(s)  \ft{u}_{\xi^{(2)}}(s), \\
\mathcal{N}_1(t) & :=   3i\sum_{\mathcal R} \sum_{\xi^{(3)} + \xi^{(4)}=\xi^{(1)}} \int_0^t  e^{-is \Phi_{234}} \frac{1}{(\al \cdot \xi^{(2)})} \ft{u}_{\xi^{(2)}}(s)  \ft{u}_{\xi^{(3)}}(s)   \ft{u}_{\xi^{(4)}}(s) \, ds , \nonumber\\
\mathcal{N}_2(t) & :=  3i \int_0^t \sum_{\mathcal R} \sum_{\xi^{(3)} + \xi^{(4)}=\xi^{(2)}} e^{-is \Phi_{134}} \frac{1}{(\al \cdot \xi^{(1)})} \ft{u}_{\xi^{(1)}}(s)  \ft{u}_{\xi^{(3)}}(s)   \ft{u}_{\xi^{(4)}}(s) \, ds , \nonumber
\end{align*}
with $\Phi_{jk\l} = 3 [\al \cdot(\xi^{(j)} + \xi^{(k)})] [\al \cdot(\xi^{(j)} + \xi^{(\l)})] [\al \cdot(\xi^{(k)} + \xi^{(\l)})]$.  We will start by estimating the boundary term $\mathcal{B}$ and the nonlinear term $\mathcal{N}_2$. We will need further case separation to estimate $\mathcal{N}_1$.

\smallskip

\noi\underline{Estimating $\mathcal{B}$}. Applying the Cauchy--Schwarz inequality, we get
\begin{align*}
\| \mathcal{B}(s)\|_{\l^1_\xi} & \les \sum_\xi \sum_{\mathcal R} \frac{1}{|\al\cdot \xi^{(1)}| | \al \cdot \xi^{(2)}|} |\ft{u}_{\xi^{(1)}}(s) \ft{u}_{\xi^{(2)}}(s) | \les B^\frac12 \|u\|^2_{L^\infty G^\s}
\end{align*}
where
\begin{align*}
B:= \sum_\xi \sum_{\xi = \xi^{(1)} + \xi^{(2)}} \frac{1}{|\al\cdot \xi^{(1)}| | \al \cdot \xi^{(2)}| \jb{\xi_1^{(1)}}^{2\theta } \jb{\xi_2^{(1)}}^{2\theta } \jb{\xi_1^{(2)}}^{2\theta } \jb{\xi_2^{(2)}}^{2\theta }} .
\end{align*}
Since $\|u\|_{L^\infty G^\s} = \|q\|_{L^\infty G^\s}$, it only remains to show that $B\les 1$. From \eqref{diophantine} and the fact that $|\al\cdot \xi^{(1)}| \sim |\xi_1^{(1)}| \gg |\xi_1^{(2)}|+  |\xi_2^{(2)}|$, which follows from the restriction to $\mathcal R$, we have 
\begin{align*}
B & \les \sum_\xi \sum_{\mathcal R} \frac{|\xi^{(2)}|^{\gamma}}{ \jb{\xi_1^{(1)}}^{2\theta +1 } \jb{\xi_2^{(1)}}^{2\theta } \jb{\xi_1^{(2)}}^{2\theta } \jb{\xi_2^{(2)}}^{2\theta }} \\
&\les \!\! \sum_{\xi^{(1)}, \xi^{(2)}} \!\!\frac{1}{ \jb{\xi_1^{(1)}}^{2\theta +1-\gamma} \jb{\xi_2^{(1)}}^{2\theta } \jb{\xi_1^{(2)}}^{2\theta } \jb{\xi_2^{(2)}}^{2\theta }}  \les 1, 
\end{align*}
provided $2\theta +1 -\gamma >1$, or equivalently, $\theta >\frac{\gamma}{2}$.

\medskip 

\noi\underline{Estimating $\mathcal{N}_2$}. From the Cauchy--Schwarz inequality we have
\begin{align*}
\big\|\mathcal{N}_2(t) \big\|_{\l^1_\xi} & \les \int_0^t \sum_\xi \sum_{\mathcal R} \sum_{\xi^{(3)} + \xi^{(4)}=\xi^{(2)}} \frac{1}{|\al \cdot \xi^{(1)}|} |\ft{u}_{\xi^{(1)}}(s)  \ft{u}_{\xi^{(3)}}(s)   \ft{u}_{\xi^{(4)}}(s) | \, ds\\
&\les   C^\frac12  T \|u\|^3_{L^\infty G^\s},
\end{align*}
where 
\begin{align*}
C:=&\sum_\xi \sum_{\mathcal R} \sum_{\xi^{(3)} + \xi^{(4)}=\xi^{(2)}} \frac{|\al\cdot \xi^{(3)}| |\al \cdot \xi^{(4)}| }{|\al \cdot \xi^{(1)}| \prod\limits_{j\in\{1,3,4\}} \jb{\xi_1^{(j)}}^{2\theta} \jb{\xi_2^{(j)}}^{2\theta} }. \\
{}\les& \sum_\xi \sum_{\mathcal R}  \sum_{\xi^{(3)} + \xi^{(4)}=\xi^{(2)}} \frac{1}{\jb{\xi_1^{(1)}}^{2\theta+1} \jb{\xi_2^{(1)}}^{2\theta}
	\jb{\xi_1^{(3)}}^{2\theta-1}\jb{\xi_1^{(4)}}^{2\theta-1}\jb{\xi_2^{(3)}}^{2\theta-1}\jb{\xi_2^{(4)}}^{2\theta-1} } .
\end{align*}
On the region of summation we have
$$
\jb{\xi_1^{(1)}} \gg \jb{ \xi^{(2)} } = \jb{ \xi^{(3)} + \xi^{(4)} } \gtrsim \jb{ \xi^{(3)}_1 + \xi^{(4)}_1 }^{\frac12}\jb{ \xi^{(3)}_2 + \xi^{(4)}_2 } ^{\frac12} .
$$
Combining this with the Hardy--Littlewood (weak-Young) inequality, we deduce that
\begin{align*}
C & \les \sum_\xi \sum_{\mathcal R}  \sum_{\xi^{(3)} + \xi^{(4)}=\xi^{(2)}} \frac{1}{\jb{\xi_1^{(1)}}^{2\theta} \jb{\xi_2^{(1)}}^{2\theta} }
	\prod_{\ell=1}^2 \frac{1}{ \jb{\xi_\ell^{(3)}}^{2\theta-1} \jb{\xi_\ell^{(3)}+\xi_\ell^{(4)}}^{\frac12} \jb{\xi_\ell^{(4)}}^{2\theta-1} } \lesssim 1, 
\end{align*}
provided $2\theta>1$ and $\frac43(2\theta-1)>1$, or equivalently, $\theta > \frac78$.

\medskip 

\noi\underline{Estimating $\mathcal{N}_1$}.  The arguments used to estimate $\mathcal{B}$ and $\mathcal{N}_2$ are insufficient to estimate $\mathcal{N}_1$ due to the negative power of $|\al\cdot \xi^{(2)}|$ which can be small. In this case, we rewrite $\mathcal{N}_1$ in the variables $q$ instead of $u$. This gives
\begin{align}\label{nonlinear1-q}
\mathcal{N}_1(t) & = 3i\sum_{\mathcal R} \sum_{\xi^{(3)} + \xi^{(4)}=\xi^{(1)}} \int_0^t e^{-is (\al \cdot \xi)^3} \frac{1}{(\al \cdot \xi^{(2)})} \ft{q}_{\xi^{(2)}}(s)  \ft{q}_{\xi^{(3)}}(s)   \ft{q}_{\xi^{(4)}}(s) \, ds. 
\end{align}

From here, we argue in a manner closer to that used in Case~1.  First, using the factorization \eqref{factor}, under the assumption $\tau_2+\tau_3+\tau_4 = (\al\cdot \xi)^3$ we obtain
\begin{align*}
\max_{j=2,3,4}|\tau_j - (\al\cdot \xi^{(j)})^3| &\ges | \tau_2 - (\al\cdot \xi^{(2)})^3 + \tau_3 - (\al\cdot \xi^{(3)})^3  + \tau_4 - (\al\cdot \xi^{(4)})^3 | \\
& =  | (\al\cdot \xi)^3 - (\al\cdot \xi^{(2)})^3  - (\al\cdot \xi^{(3)})^3 - (\al\cdot \xi^{(4)})^3|\\
& = 3 |\al\cdot(\xi^{(2)} + \xi^{(3)}) |  |\al\cdot(\xi^{(2)} + \xi^{(4)}) | |\al\cdot(\xi^{(3)} + \xi^{(4)}) | \\
& = |\Phi_{234}|.
\end{align*}
We present the details in the case 
$$
\max_{j=2,3,4}|\tau_j - (\al\cdot \xi^{(j)})^3| = |\tau_3 - (\al\cdot \xi^{(3)})^3|.
$$
When the maximum is $|\tau_j - (\al\cdot \xi^{(j)})^3|$ for $j=2$ or $j=4$, one simply swaps the roles of $\ft{q}_{\xi^{(3)}}$ and $\ft{q}_{\xi^{(j)}}$ when handling the time integral, including which term $\chi_{[0,t]}$ gets grouped with.

Focusing on the time integral in \eqref{nonlinear1-q}, we have
\begin{align*}
 \int_0^t  & e^{-is (\al \cdot \xi)^3}  \ft{q}_{\xi^{(2)}}(s)  \ft{q}_{\xi^{(3)}}(s)   \ft{q}_{\xi^{(4)}}(s) \, ds \\
& = \int_\R e^{-is (\al \cdot \xi)^3} \chi_{[0,t]}(s) \ft{q}_{\xi^{(2)}}(s)  \ft{q}_{\xi^{(3)}}(s)   \ft{q}_{\xi^{(4)}}(s) \, ds\\
& =  \int \delta\bigl( (\al\cdot \xi)^3 - \tau_2 -\tau_3 -\tau_4\bigr)  \F_t\big(\chi_{[0,t]}  \ft{q}_{\xi^{(2)}}  \big) (\tau_2) \cdot \F_t \ft{q}_{\xi^{(3)}}(\tau_3)  \cdot \F_t \ft{q}_{\xi^{(4)}}(\tau_4) \, d\vec{\tau}.
\end{align*}
With $0<b<\frac12$, we use Cauchy--Schwarz and Lemma~\ref{lm:cutoff} to estimate
\begin{align*}
&\bigg|  \int_0^t  e^{-is (\al \cdot \xi)^3}  \ft{q}_{\xi^{(2)}}(s)  \ft{q}_{\xi^{(3)}}(s)   \ft{q}_{\xi^{(4)}}(s) \, ds \bigg|\\
& \les \frac{1}{\jb{\Phi_{234}}^\frac12} \int \frac{\delta\bigl( (\al\cdot \xi)^3-\tau_2 -\tau_3 -\tau_4\bigr)}{\jb{\tau_2 - (\al\cdot \xi^{(2)})^3}^{b} \jb{\tau_4 - (\al \cdot \xi^{(4)})^3}^{\frac12}} \\
& \phantom{xxxxxx}\times \jb{\tau_2 - (\al\cdot \xi^{(2)})^3}^{b} \big|\F_t\big(\chi_{[0,t]}  \ft{q}_{\xi^{(2)}}  \big) (\tau_2) \big| \, \prod_{j=3}^4 \jb{\tau_j - (\al\cdot \xi^{(j)})^3}^\frac12 
\big| \F_t \ft{q}_{\xi^{(j)}}(\tau_j) \big| \, d\vec{\tau} \\
& \les \frac{1}{\jb{\Phi_{234}}^\frac12} \bigg[ \int_{\R^2} \frac{1}{\jb{(\al\cdot \xi)^3 - \tau_3 - \tau_4 - (\al\cdot \xi^{(2)})^3}^{2b} \jb{\tau_4 - (\al \cdot \xi^{(4)})^3}} \\
& \phantom{xxxxxx}\times  \big| \jb{\tau_3 - (\al\cdot \xi^{(3)})^3}^\frac12 
 \F_t \ft{q}_{\xi^{(3)}}(\tau_3) \big|^2  d\tau_3 \,d\tau_4\bigg]^\frac12 \\
 & \phantom{xxxxxx} \times \big\| \jb{\tau - (\al\cdot \xi^{(2)})^3}^{b} \F_t\big(\chi_{[0,t]}  \ft{q}_{\xi^{(2)}}  \big) (\tau) \big\|_{L^2_\tau} \big\| \jb{\tau - (\al\cdot \xi^{(4)})^3}^\frac12 \F_t\ft{q}_{\xi^{(4)}} (\tau) \big\|_{L^2_\tau} \\
 & \les \frac{1}{\jb{\Phi_{234}}^\frac12} \prod_{j=2}^4 \big\| \jb{\tau - (\al\cdot \xi^{(j)})^3}^\frac12 \F_t\ft{q}_{\xi^{(j)}} (\tau) \big\|_{L^2_\tau}.
\end{align*}
This yields
\begin{align*}
\big\| \mathcal{N}_1 (t) \big\|_{\l^1_\xi} & \les  \sum_\xi \sum_{\mathcal R} \sum_{\xi^{(3)} + \xi^{(4)}=\xi^{(1)}}  \frac{1}{|\al \cdot \xi^{(2)}| \jb{\Phi_{234}}^\frac12} \prod_{j=2}^4 \big\| \jb{\tau - (\al\cdot \xi^{(j)})^3}^\frac12 \F_t\ft{q}_{\xi^{(j)}} (\tau) \big\|_{L^2_\tau}\\
& \les D^\frac12 \|q \|^3_{X^{\theta, \frac12}}
\end{align*}
where
\begin{align*}
D &:= \sum_\xi \sum_{\mathcal R} \sum_{\xi^{(3)} + \xi^{(4)}=\xi^{(1)}}  \frac{|\al\cdot \xi^{(3)}| |\al \cdot \xi^{(4)}| }{|\al\cdot \xi^{(2)}| \jb{\Phi_{234}} \jb{\xi_1^{(2)}}^{2\theta} \jb{\xi_2^{(2)}}^{2\theta} \jb{\xi_1^{(3)}}^{2\theta} \jb{\xi_2^{(3)}}^{2\theta} \jb{\xi_1^{(4)}}^{2\theta} \jb{\xi_2^{(4)}}^{2\theta} }.
\end{align*}

To complete the proof, it remains to show that $D\lesssim 1$. By symmetry, we merely need to estimate the part of the sum where
\begin{align}\label{sym ass'}
|\al\cdot \xi^{(3)}| \geq |\al\cdot \xi^{(4)}|.
\end{align}
We decompose into further regions depending on the size of $|\al\cdot \xi^{(2)}|$.

\medskip 

\noi\textbf{Case 2.1: $ |\al\cdot \xi^{(2)} | \ges |\al \cdot \xi^{(4)}|$} and $\max\limits_{j=1,2} |\xi_j^{(3)}| \les \min\limits_{j=1,2}|\xi_j^{(3)}| + |\xi_1^{(2)}| + |\xi_2^{(2)}| +|\xi_1^{(4)}| + |\xi_2^{(4)}|$.  
In this case, we do not need to exploit the factor $\jb{\Phi_{234}}$ appearing in the denominator.  Indeed, the contribution of this part of the sum to $D$ is bounded by a constant multiple of
\begin{align*}
 &\sum_{\xi^{(2)}, \xi^{(3)}, \xi^{(4)}} \frac{|\al\cdot \xi^{(3)}|}{\jb{\xi_1^{(2)}}^{2\theta} \jb{\xi_2^{(2)}}^{2\theta} \jb{\xi_1^{(3)}}^{2\theta} \jb{\xi_2^{(3)}}^{2\theta} \jb{\xi_1^{(4)}}^{2\theta} \jb{\xi_2^{(4)}}^{2\theta} } \\
&\quad \les \sum_{\xi^{(2)}, \xi^{(3)}, \xi^{(4)}} \frac{1}{\jb{\xi_1^{(2)}}^{2\theta-\frac12} \jb{\xi_2^{(2)}}^{2\theta-\frac12} \jb{\xi_1^{(3)}}^{2\theta-\frac12} \jb{\xi_2^{(3)}}^{2\theta-\frac12} \jb{\xi_1^{(4)}}^{2\theta-\frac12} \jb{\xi_2^{(4)}}^{2\theta-\frac12} } \les 1,
\end{align*}
provided $2\theta-\frac12 > 1$, or equivalently, $\theta> \frac34$. 

\medskip 

\noi\textbf{Case 2.2: $ |\al\cdot \xi^{(2)} | \ges |\al \cdot \xi^{(4)}|$} and $\max\limits_{j=1,2} |\xi_j^{(3)}| \gg \min\limits_{j=1,2}|\xi_j^{(3)}|+  |\xi_1^{(2)}|+  |\xi_2^{(2)}|+|\xi_1^{(4)}|+  |\xi_2^{(4)}|$.
\noi Recalling \eqref{sym ass}, we see that $|\xi_1^{(3)}|\gg|\xi_2^{(3)}|+ |\xi^{(2)}|+ |\xi^{(4)}|$ and
\begin{align*}
|\Phi_{234}| & \sim |\al\cdot (\xi^{(2)} + \xi^{(3)})| |\al\cdot (\xi^{(2)} + \xi^{(4)})| |\al\cdot (\xi^{(3)} + \xi^{(4)})| \\
&\sim |\al \cdot \xi^{(3)} |^2 | \al\cdot(\xi^{(2)} + \xi^{(4)})|. 
\end{align*}
Using also the diophantine condition \eqref{diophantine}, we may bound the contribution of this part of the sum to $D$ by a constant multiple of
\begin{align*}
\sum_{\xi^{(2)}, \xi^{(3)}, \xi^{(4)}} & \frac{1 }{ |\al\cdot \xi^{(3)}| |\al\cdot (\xi^{(2)} + \xi^{(4)})| \prod\limits_{j=2}^4 \jb{\xi_1^{(j)}}^{2\theta} \jb{\xi_2^{(j)}}^{2\theta}  } \\
& \les \sum_{\xi^{(2)}, \xi^{(3)}, \xi^{(4)}}  \frac{|\xi^{(2)} + \xi^{(4)}|^\gamma}{ \jb{\xi_1^{(3)}}^{2\theta+1} \jb{\xi_2^{(3)}}^{2\theta} \jb{\xi_1^{(2)}}^{2\theta} \jb{\xi_2^{(2)}}^{2\theta} \jb{\xi_1^{(4)}}^{2\theta} \jb{\xi_2^{(4)}}^{2\theta}  }\\
& \les \sum_{\xi^{(2)}, \xi^{(3)}, \xi^{(4)}}\frac{1}{ \jb{\xi_1^{(3)}}^{2\theta+1-\gamma} \jb{\xi_2^{(3)}}^{2\theta} \jb{\xi_1^{(2)}}^{2\theta} \jb{\xi_2^{(2)}}^{2\theta} \jb{\xi_1^{(4)}}^{2\theta} \jb{\xi_2^{(4)}}^{2\theta} }\lesssim 1,
\end{align*}
provided $2\theta + 1 - \gamma>1$, or equivalently, $\theta > \frac\gamma2$.

\medskip 

\noi{\bf Case 2.3: $|\al \cdot \xi^{(2)} | \ll |\al \cdot \xi^{(4)}|$}. Recalling \eqref{sym ass'}, on this region we have
\begin{align*}
|\Phi_{234}| \sim |\al\cdot \xi^{(3)}| |\al\cdot \xi^{(4)}| | \al \cdot (\xi^{(3)}+\xi^{(4)}) |.
\end{align*}
As we are working in the region $\mathcal R$, we have $| \al \cdot (\xi^{(3)}+\xi^{(4)}) | =| \al \cdot \xi^{(1)} |\gg |\xi^{(2)} |\gtrsim \jb{\xi_1^{(2)}}^\frac12\jb{\xi_2^{(2)}}^\frac12$. Using also \eqref{diophantine}, we may bound the contribution of this part of the sum to $D$ by a constant multiple~of
\begin{align*}
\sum_\xi \sum_{\mathcal R} &\sum_{\xi^{(3)} + \xi^{(4)}=\xi^{(1)}}  \frac{1}{|\al\cdot \xi^{(2)}|  | \al \cdot (\xi^{(3)}+\xi^{(4)}) | \prod\limits_{j=2}^4 \jb{\xi_1^{(j)}}^{2\theta} \jb{\xi_2^{(j)}}^{2\theta} }\\
& \les \sum_{\xi^{(2)}, \xi^{(3)}, \xi^{(4)}}  \frac{1}{|\xi^{(2)}|^{1-\gamma} \prod\limits_{j=2}^4 \jb{\xi_1^{(j)}}^{2\theta} \jb{\xi_2^{(j)}}^{2\theta} }\\
& \les \sum_{\xi^{(2)}, \xi^{(3)}, \xi^{(4)}}  \frac{1}{\jb{\xi_1^{(2)}}^{2\theta+\frac12-\frac\gamma2}\jb{\xi_2^{(2)}}^{2\theta+\frac12-\frac\gamma2} \prod\limits_{j=3}^4 \jb{\xi_1^{(j)}}^{2\theta} \jb{\xi_2^{(j)}}^{2\theta} } \les 1,
\end{align*}
provided $2\theta + \frac12 - \frac\gamma2>1$, which is implied by $\theta > \frac\gamma2$ as $\gamma>1$.

This completes the treatment of $\mathcal N_1$ and so the proof of the proposition.
\end{proof}

\section{A counter-example to the Deift conjecture}

In this section we prove Theorem~\ref{th:deift}. We first recall the following special case of results due to Oskolkov \cite{oskolkov} on trigonometric sums with polynomial phases:
 
\begin{theorem}[{\cite[Proposition~12]{oskolkov}}]\label{th:airy}
Fix $\alpha>0$ and let $f$ be a periodic function with period $2\pi \al^{-1}$ that is of bounded total variation over the period. Then the solution 
$$
w(t,x)= \sum_{\xi\in\Z} \ft{f}(\xi) e^{i(\al \xi x + \al^3 \xi^3 t)}
$$
to the Airy equation \eqref{Airy} with initial data $f$ has the following properties:\\[1mm]
{\rm(i)} $w \in L^\infty_{t,x}$ and $\|w\|_{L^\infty_{t,x}}\les \| f\|_{L^\infty} + \operatorname{var}(f)$; \\[1mm]
{\rm(ii)} The set of discontinuities of $w$ in $(t,x) \in \R^2$ is countable;\\[1mm]
{\rm(iii)} For each $t$ such that $\frac{\al t }{2\pi}$ is irrational, $w(t)$ is a continuous function of $x$.
\end{theorem}

It follows from Theorem~\ref{th:airy} that the Deift Conjecture fails for the Airy equation.  Indeed, fix $\al=(\al_1, \al_2)\in \R^2$ satisfying \eqref{6:59} and let $f(x) = f_1(x) + f_2(x)$ where $f_1(x)= \operatorname{sq}(\al_1x)$ and $f_2(x)= \operatorname{sq}(\al_2x)$, which are both periodic functions of bounded total variation.  The solution to the Airy equation \eqref{Airy} with initial data $f$ can be written as 
\begin{align*}
w(t,x) = \sum_{\xi\in \Z} \ft{f}_1(\xi) e^{i(\al_1 \xi x + \al_1^3 \xi^3 t) } + \sum_{\xi\in \Z} \ft{f}_2(\xi) e^{i(\al_2 \xi x + \al_2^3 \xi^3 t) }.
\end{align*} 
By Theorem~\ref{th:airy}, $w$ is a bounded function of spacetime; moreover, $w(t_0,x)$ is a continuous function of $x$ for every time $t_0$ such that $\frac{\al_1 t_0}{2\pi}$ and $\frac{\al_2 t_0}{2 \pi}$ are both irrational. For such times $t_0$, $w(t_0)$ is the sum of two continuous periodic functions and so is almost periodic. However, the initial data $w(0)=f$ is not almost periodic: it is not continuous and it does not have dense almost periods.

Building on this observation and the analysis in the previous section, we are now ready to prove that the Deift Conjecture also fails for the Korteweg--de Vries equation.

\begin{proof}[Proof of Theorem~\ref{th:deift}]\
Fix $\al = (\al_1,\al_2) \in\R^2$ satisfying the Diophantine condition \eqref{diophantine} for some fixed $C_0>0$ and $\gamma >1$, and let
$$
f(x) = \operatorname{sq}(\al_1x) + \operatorname{sq}(\al_2x).
$$
Clearly, $f$ is bounded and of bounded total variation.  As remarked above, $f$ is not almost periodic: it is not continuous and it does not have dense almost periods.

By Theorem~\ref{th:lwp}, there exist $T>0$ and a solution $u$ of KdV on the interval $(-T,T)$ with initial data $f$.  According to Theorem~\ref{th:airy}(i), the solution $w(t)=e^{-t\partial_x^3}f $ to the Airy equation with initial data $f$ is bounded globally in spacetime.  On the other hand, Theorem~\ref{prop:nonlinear-smoothing} shows that the nonlinear part $u(t,x) - w(t,x)$ is bounded on $(-T,T)\times\R$.  Thus $u(t,x)$ is bounded throughout $(-T,T)\times\R$.  

By Theorem~\ref{th:lwp}, $u:(-T,T)\to G^\theta$ is continuous and so $t\mapsto\langle \varphi, u(t)\rangle$ is continuous on $(-T,T)$ for any Schwartz function $\varphi$.  Combining this with the fact that $u(t,x)$ is bounded on $(-T,T)\times\R$, we see that $u:(-T,T)\to L^\infty(\R)$ is weak-$\ast$ continuous.

In view of Theorem~\ref{th:unique}, we therefore deduce that $u$ is the \emph{unique} solution of KdV with initial data $f$ consistent with Definition~\ref{def:solution}.
 
By Theorem~\ref{prop:nonlinear-smoothing}, the nonlinear part $u(t) - w(t)$ is a continuous almost periodic function of space for every $t \in (-T,T)$.  Moreover, Theorem~\ref{th:airy}(iii) guarantees that $w(t_0)$ is also a continuous almost periodic function of space for any $t_0 \in (-T,T)$ for which $\frac{\al_1 t_0}{2\pi}$ and $\frac{\al_2 t_0}{2\pi}$ are both irrational.  For such $t_0$ we conclude that $u(t_0)$ is a continuous almost periodic function of space. 

Fixing such $t_0$ that is negative, we choose $q_0 := u(t_0)$ as our almost periodic initial data. The unique solution with this initial data is, of course, $q: t\mapsto u(t+t_0)$. In this way, we see that the KdV evolution carries the initial data $q_0$ that is almost periodic to the state $q(|t_0|)= f$ that is \emph{not} almost periodic. 
\end{proof}


\section{The case against Stepanov almost periodicity}\label{sec:stepanov}

In the previous section, we considered a solution of \eqref{kdv} with almost periodic initial data, whose evolution developed discontinuities and so left the class of almost periodic functions.  Although this breakdown of almost periodicity was very mild, it allowed us to give a rather precise description of the evolution; it also allowed us to guarantee the uniqueness of our solution.

Given the mild nature of this breakdown, it is tempting to believe that a wider notion of almost periodicity, for example, Stepanov-almost periodicity, may change the answer.  Given $1\leq p< \infty$, the corresponding class is defined as the closure of the set of trigonometric polynomials under the norm
$$
\| f \|_{S^p}^p = \sup_y \int_{-1}^{1} \bigl| f(x+y) \bigr|^p \, dx.
$$
For $p=\infty$, we recover the original notion of an almost periodic function.

The goal of this section is to argue that such remedies are illusory; they do not even resolve the case of the Airy evolution
\begin{equation}\label{Airy'}
\tfrac{d}{dt} q = - q'''  .
\end{equation}
Concretely, we will give an example of almost periodic initial data (in the sense of Bohr and so also Stepanov), whose evolution under the Airy flow undergoes an infinite concentration of $L^2$ norm in finite time.

In the example below, the wave packets that come together are well-separated in frequency.  For this reason, we can expect their evolutions to interact only very weakly under the full KdV evolution.  Indeed, the frequency and spatial separations of the wave packets can be increased tremendously, without compromising the analysis below in any way.   However, given that the entire point of this section is to demonstrate what we regard as a fatally flawed direction for further investigation, we refrain from analyzing the full KdV evolution. 

\begin{proposition}
There is a bounded almost periodic function $u_0:\R\to\R$ whose evolution through time $t_0$ under \eqref{Airy'} satisfies
\begin{align}\label{infinite L2'}
\int_\R | u(t_0,x) |^2 e^{-x^2} \,dx = \infty.
\end{align}
\end{proposition}

\begin{proof}
Let us fix $t_0= 10^{-6}$ once and for all.  For each frequency parameter $\eta\in \R$, we consider the complex-valued solution $\phi(t,x; \eta)$ to the Airy equation \eqref{Airy'} with initial data
\begin{align}
\widehat \phi(0,\xi; \eta) = \exp\bigl\{ - i t_0 \eta^3 -3it_0\eta^2(\xi-\eta) - 3it_0\eta(\xi-\eta)^2 - \tfrac12 (\xi-\eta)^2\bigr\} .
\end{align}
Evidently,
\begin{align}\label{0529}
\phi(0,x; \eta) = (1 + 6it_0\eta)^{-\frac12} \exp\bigl\{ - i t_0\eta^3 + i\eta x  -\tfrac12(1+6it_0\eta)^{-1}[x-3t_0\eta^2]^2 \bigr\}
\end{align}
and
\begin{align}\label{0530}
\widehat \phi(t_0,\xi; \eta) = \exp\bigl\{ i t_0 (\xi-\eta)^3 -\tfrac12 (\xi-\eta)^2 \bigr\} .
\end{align}

It is clear that the functions presented in \eqref{0529} and \eqref{0530} are Schwartz; however, we need additional quantitative information in order to assemble our initial data $u_0$. Direct computation shows that
\begin{align}\label{0531}
\bigl[ 1 + \tfrac{(x - 3t_0\eta^2)^2}{1+t_0^2\eta^2} \bigr] \bigl|\phi(0,x; \eta) \bigr|   \lesssim ( 1 + t_0^2\eta^2 )^{-\frac14}
\end{align}
uniformly in $x$ and $\eta$.  Regarding the solutions at time $t_0$, we have
\begin{align}\label{0532}
\bigl\langle \phi(t_0,x; \eta) , e^{-x^2} \phi(t_0,x;\eta)\bigr\rangle_{L^2_x(\R)} \gtrsim 1
\end{align}
and
\begin{align}\label{0533}
\bigl\langle \phi(t_0,x-y'; \eta') , e^{-x^2} \phi(t_0,x-y;\eta)\bigr\rangle_{L^2_x(\R)}\lesssim \langle y \rangle^{-2} \langle y' \rangle^{-2}  \langle\eta-\eta' \rangle^{-2}
\end{align}
uniformly for $\eta,\eta' \in \R$.

The parameter $t_0$ was chosen so small precisely to make the verification of \eqref{0532} easy.  Indeed, setting $t_0=0$ in RHS\eqref{0530} gives
$$
\widehat \psi(\xi; \eta) = \exp\bigl\{ -\tfrac12 (\xi-\eta)^2 \bigr\}, \qtq{which implies} \psi(x; \eta) = \exp\bigl\{ix\eta -\tfrac12x^2 \bigr\} .
$$
For this function we may compute the inner product explicitly:
$$
\bigl\langle  \psi(x; \eta) , e^{-x^2}  \psi(x; \eta)\bigr\rangle_{L^2_x(\R)} =\tfrac12\sqrt{2\pi}.
$$
On the other-hand, by Cauchy--Schwarz and Plancherel we may estimate
\begin{align*}
\Bigl|\bigl\langle \phi(t_0,x; \eta) , \, &e^{-x^2} \phi(t_0,x;\eta)\bigr\rangle_{L^2_x(\R)} -\bigl\langle  \psi(x; \eta) , e^{-x^2}  \psi(x; \eta)\bigr\rangle_{L^2_x(\R)}\Bigr|\\
&\lesssim \bigl\|\phi(t_0,x; \eta)- \psi(x; \eta) \bigr\|_{L^2_x(\R)}\lesssim |t_0| \bigl\|\xi^3 \exp\bigl\{-\tfrac12\xi^2 \bigr\}\bigr\|_{L^2_\xi(\R)}.
\end{align*}
Thus taking $t_0$ small guarantees \eqref{0532}.

The estimate \eqref{0533} is also elementary.  Decay in $y$ and $y'$ is most easily seen in physical variables; indeed, this yields decay at an arbitrary polynomial rate.  Likewise, Gaussian-type decay in $\eta-\eta'$ can be exhibited by analyzing the inner product in Fourier variables.  The result follows by taking a geometric mean of these two estimates.

Let us now define our initial data:
\begin{align}\label{0534}
u_0(x) = \sum_{n \in \mathbb N} \sum_{k\in\Z} \sum_\pm  \phi(0,x - k 2^{n}; \pm 2^n) .
\end{align}
For fixed $n$, we see that the inner two sums yield a smooth real-valued $2^n$-periodic function whose supremum norm is of size $O(2^{-n/2})$; see \eqref{0531}.  In this way, we see that $u_0$ is indeed almost periodic.  In fact, because the periods are commensurate, we see that $u_0$ is a uniform limit of continuous periodic functions; such functions are said to be limit periodic.

We now consider the corresponding solution to \eqref{Airy'} at time $t_0$.  Our goal is to show \eqref{infinite L2'}.
Expanding $u(t_0,x)$ as in \eqref{0534}, we may consider two collections of terms, namely, the diagonal terms
\begin{align*}
\sum_{n \in \mathbb N} \sum_{k\in\Z} \sum_\pm  \bigl\langle \phi(t_0,x - k2^{n}; \pm 2^n) , e^{-x^2} \phi(t_0,x - k 2^{n}; \pm 2^n)\bigr\rangle_{L^2_x(\R)}
\end{align*}
and the off-diagonal terms
\begin{align*}
\sum \bigl\langle \phi(t_0,x-y; \eta) , e^{-x^2} \phi(t_0,x-y';\eta')\bigr\rangle_{L^2_x(\R)},
\end{align*}
where the sum is over all choices
\begin{align*}
\Bigl\{ (y,\eta,y',\eta') :{} & |\eta| \in 2^{ \mathbb N},\ |\eta'| \in 2^{ \mathbb N},\ y=k\eta, \text{ and } y'=k'\eta', \text{ for some }k,k'\in\Z \\
&\text{subject to } \eta\neq\eta' \text{ or } y\neq y' \Bigr\}.
\end{align*}
The estimate \eqref{0533} guarantees the contribution of these off-diagonal terms is absolutely summable. The diagonal contribution clearly diverges by virtue of \eqref{0532}.  This proves \eqref{infinite L2'}.
\end{proof}


\begin{thebibliography}{99}
%

\bibitem{MR3548255}
K. Andreiev, I. Egorova, T. L. Lange, G. Teschl,
\emph{Rarefaction waves of the Korteweg--de Vries equation via nonlinear steepest descent.}
J. Differential Equations 261 (2016), no. 10, 5371--5410. 


\bibitem{MR2789490}
A. V. Babin, A. A. Ilyin, E. S. Titi,
\emph{On the regularization mechanism for the periodic Korteweg-de Vries equation.} 
Comm. Pure Appl. Math. 64 (2011), no. 5, 591--648. 

\bibitem{BerryKlein}
M. V. Berry, S. Klein,
\emph{Integer, fractional and fractal Talbot effects.}
J. Mod. Optics, 43 (1996), no. 10, 2139--2164.

\bibitem{BMS2022}
M. V. Berry, I. Marzoli, W. Schleich,
\emph{Quantum carpets, carpets of light.}
Physics World, 14 (2001), no. 6, 39--44.




\bibitem{BDGL18} I. Binder, D. Damanik, M. Goldstein, M. Lukic,
\emph{Almost periodicity in time of solutions of the KdV equation.} Duke Math. J. 167 (2018), no. 14, 2633--2678.


\bibitem{Bohr:book}
H.~Bohr,
\emph{Almost Periodic Functions.}
Chelsea Publishing Co., New York, N.Y., 1947.




\bibitem{Bo93-2}
J.~Bourgain,
\emph{Fourier transform restriction phenomena for certain lattice subsets and applications to nonlinear evolution equations. II. The KdV equation},
Geom. Funct. Anal. 3 (1993), no.~3, 209--262.

\bibitem{MR0139371}
V. Buslaev, V Fomin,
\emph{An inverse scattering problem for the one-dimensional Schr\"odinger equation on the entire axis.}
Vestnik Leningrad Univ. 17 (1962), no. 1, 56--64. 

\bibitem{MR3094557}
G. Chen, P. J. Olver,
\emph{Numerical simulation of nonlinear dispersive quantization.}
Discrete Contin. Dyn. Syst. 34 (2014), no. 3, 991--1008. 

\bibitem{arXiv:0503366}
M. Christ,
\emph{Nonuniqueness of weak solutions of the nonlinear Schr\"odinger equation.}
Preprint 2005, \texttt{arXiv:0503366}.



\bibitem{MR0748367}
A. Cohen,
\emph{Solutions of the Korteweg-de Vries equation with steplike initial profile.}
Comm. Partial Differential Equations 9 (1984), no. 8, 751--806. 

\bibitem{MR0988885}
A. Cohen, T. Kappeler,
\emph{Nonuniqueness for solutions of the Korteweg-de Vries equation.} 
Trans. Amer. Math. Soc. 312 (1989), no. 2, 819--840.


\bibitem{Damanik_ICERM}
D. Damanik,
\emph{Solutions to the KdV and related equations with almost periodic initial data.}
Recorded lecture December 8, 2021, at the conference ``Hamiltonian Methods and Asymptotic Dynamics'' hosted by ICERM.


\bibitem{DG16}
D. Damanik, M. Goldstein, 
\emph{On the existence and uniqueness of global solutions for the KdV equation with quasi-periodic initial data.} 
J. Amer. Math. Soc. 29 (2016), no. 3, 825--856.


\bibitem{DLVY21} 
D.~Damanik, M.~Luki\'{c}, A.~Volberg, P.~Yuditskii,
\emph{The Deift conjecture: a program to construct a counter-example.} Preprint 
\texttt{arXiv:2111.09345}.

\bibitem{Deift08}
P. Deift,
\emph{Some open problems in random matrix theory and the theory of integrable systems.}
Contemp. Math., vol. 458, Amer. Math. Soc., Providence, RI, 2008, pp. 419--430.

\bibitem{Deift17}
P. Deift,
\emph{Some open problems in random matrix theory and the theory of integrable systems, II.}
SIGMA 13 (2017), 016.


\bibitem{MR1263128}
P. Deift, S. Venakides, X. Zhou,
\emph{The collisionless shock region for the long-time behavior of solutions of the KdV equation.}
Comm. Pure Appl. Math. 47 (1994), no. 2, 199--206.

%


\bibitem{MR0382877}
B. A. Dubrovin, S. P. Novikov, 
\emph{Periodic and conditionally periodic analogs of the many-soliton solutions of the Korteweg-de Vries equation.}
Soviet Physics JETP 40 (1974), no. 6, 1058--1063; translated from 
\v{Z}. \`Eksper. Teoret. Fiz. 67 (1974), no. 6 2131--2144.

\bibitem{MR3261206}
D. Dutykh, E. Pelinovsky,
\emph{Numerical simulation of a solitonic gas in KdV and KdV-BBM equations.}
Phys. Lett. A 378 (2014), no. 42, 3102--3110. 

\bibitem{DZZ16}
S. Dyachenko, D. Zakharov, V. Zakharov,
\emph{Primitive potentials and bounded solutions of the KdV equation.}
Phys. D 333 (2016), 148--156. 



\bibitem{Egorova}
I. E. Egorova,
\emph{The Cauchy problem for the KdV equation with almost periodic initial data whose spectrum is nowhere dense.}
in ``Spectral Operator Theory and Related Topics'', Adv. Soviet Math. 19, Amer. Math. Soc., Providence, RI, 1994, 181--208. 


\bibitem{MR3071444}
I. Egorova, Z. Gladka, V. Kotlyarov, G. Teschl,
\emph{Long-time asymptotics for the Korteweg--de Vries equation with step-like initial data.}
Nonlinearity 26 (2013), no. 7, 1839--1864.

\bibitem{EVY}
B. Eichinger, T. VandenBoom, P. Yuditskii,
\emph{KdV hierarchy via Abelian coverings and operator identities.}
Trans. Amer. Math. Soc. Ser. B  6 (2019), 1--44.


\bibitem{MR4410140}
G. A. El,
\emph{Soliton gas in integrable dispersive hydrodynamics.}
J. Stat. Mech. Theory Exp. (2021), no. 11, Paper No. 114001, 69 pp.

\bibitem{MR3559154}
M. B. Erdo\u{g}an, N. Tzirakis, \emph{Dispersive partial differential equations.} 
Wellposedness and applications. London Mathematical Society Student Texts, 86. Cambridge University Press, Cambridge, 2016.


\bibitem{MR4259375}
M. Girotti, T. Grava, R. Jenkins, K. D. T.-R. McLaughlin,
\emph{Rigorous asymptotics of a KdV soliton gas.}
Comm. Math. Phys. 384 (2021), no. 2, 733--784. 


\bibitem{GP74}
A. V. Gurevich, L. P. Pitaevskii,
\emph{Nonstationary structure of a collisionless shock wave.}
Sov. Phys. JETP 38 (1974), 291--297.

\bibitem{MR2267286}
T. Kappeler, P. Topalov,
\emph{Global wellposedness of KdV in $H^{-1}(\mathbb{T},\mathbb{R})$.}
Duke Math. J. 135 (2006), no. 2, 327--360.


\bibitem{Kato95}
T. Kato,
\emph{On nonlinear Schr\"odinger equations. II. $H^s$-solutions and unconditional well-posedness.}
J. Anal. Math. 67 (1995), 281--306.






\bibitem{KMV} R. Killip, J. Murphy, M. Vi\c{s}an, 
\emph{Invariance of white noise for KdV on the line.}
Invent. Math. 222 (2020), no. 1, 203--282.


\bibitem{KV19} R. Killip, M. Vi\c{s}an, 
\emph{KdV is well-posed in $H^{-1}$.}
Ann. of Math. 190 (2019), no. 1, 249--305.


\bibitem{MR4356987} T. Laurens,
\emph{KdV on an incoming tide.} 
Nonlinearity 35 (2022), no. 1, 343--387. 


\bibitem{MR0404889}
P. D. Lax,
\emph{Almost periodic solutions of the KdV equation.} 
SIAM Rev. 18 (1976), no. 3, 351--375.



\bibitem{MR2439485}
J. A. Leach, D. J. Needham,
\emph{The large-time development of the solution to an initial-value problem for the Korteweg-de Vries equation. I. Initial data has a discontinuous expansive step.}
Nonlinearity 21 (2008), no. 10, 2391--2408.

\bibitem{MR3229496}
J. A. Leach, D. J. Needham,
\emph{The large-time development of the solution to an initial-value problem for the Korteweg--de Vries equation. II. Initial data has a discontinuous compressive step.}
Mathematika 60 (2014), no. 2, 391--414.


\bibitem{MR4122653}
M. Luki\'c, G. Young,
\emph{Uniqueness of solutions of the KdV-hierarchy via Dubrovin-type flows.}
J. Funct. Anal. 279 (2020), no. 7, 108705, 30 pp.


\bibitem{MR1098341}
V. A. Marchenko,
\emph{The Cauchy problem for the KdV equation with nondecreasing initial data.}
In ``What is integrability?'', 273--318,  Springer Ser. Nonlinear Dynam., Springer, Berlin, 1991. 
	
\bibitem{MT76}
H.~McKean, E.~Trubowitz,
\emph{Hill's operator and hyperelliptic function theory in the presence of infinitely many branch points.}
Comm. Pure. Appl. Math. 29 (1976), no. 2, 143--226.
	

	


	
\bibitem{OlverAMM}
P. J. Olver,
\emph{Dispersive quantization.}
Amer. Math. Monthly, 117 (2010), no. 7, 599--610.
	
\bibitem{oskolkov}
K.~Oskolkov, 
\emph{A class of I. M. Vinogradov's series and its applications in harmonic analysis.} In ``Progress in approximation theory (Tampa, FL, 1990)", 353--402, Springer Ser. Comput. Math., 19, Springer, New York, 1992. 
	

	
	
\bibitem{Rybkin08}
A. Rybkin,
\emph{On the evolution of a reflection coefficient under the Korteweg-de Vries flow.}
J. Math. Phys. 49 (2008), no. 7, art. ID 072701.
	
	
	


	
\bibitem{talbot}
H. F. Talbot,
\emph{Facts related to optical science, No. IV.}
Philo. Mag. 9 (1836), 401--407.
	
\bibitem{Tsugawa12}
K. Tsugawa, 
\emph{Local well-posedness of the KdV equation with quasi-periodic initial data.}
SIAM J. Math. Anal. 44 (2012), no. 5, 3412--3428.
	

\bibitem{KruskalZabusky}
N. J. Zabusky, M. D. Kruskal,
\emph{Interaction of ``solitons" in a collisionless plasma and the recurrence of initial states.}
Phys. Rev. Lett. 15 (1965), no. 15, 240--243.

	
	
	
\bibitem{ZDZ16}
D. Zakharov, S. Dyachenko, V. Zakharov,
\emph{Bounded solutions of KdV and non-periodic one-gap potentials in quantum mechanics.}
Lett. Math. Phys. 106 (2016), no. 6, 731--740. 

\bibitem{MR1440304}
Y. Zhou,
\emph{Uniqueness of weak solution of the KdV equation.}
Internat. Math. Res. Notices 1997, no. 6, 271--283. 

\end{thebibliography}
\end{document}